\documentclass[12pt,reqno]{amsart}
\usepackage[T1]{fontenc}
\usepackage[utf8]{inputenc}
\usepackage{lmodern,microtype}
\usepackage[margin=29mm]{geometry}
\usepackage{amsmath,amssymb,amsthm,mathtools}
\usepackage{tikz-cd}
\usepackage{needspace}
\usepackage{enumitem,booktabs,tabularx}
\usepackage{xcolor}
\usepackage[colorlinks=true,linkcolor=blue,citecolor=blue,urlcolor=blue]{hyperref}
\hypersetup{bookmarksdepth=2,bookmarksopen=true,bookmarksopenlevel=1}

\makeatletter
\renewcommand{\paragraph}{\@startsection{paragraph}{4}%
  \z@{.35\linespacing\@plus.15\linespacing}{-.5em}%
  {\normalfont\bfseries}}
\makeatother
\numberwithin{equation}{section}
\newtheorem{theorem}{Theorem}[section]
\newtheorem{proposition}[theorem]{Proposition}
\newtheorem{lemma}[theorem]{Lemma}
\newtheorem{corollary}[theorem]{Corollary}
\newtheorem{mainresult}{Main result}

\theoremstyle{definition}

\newtheorem{example}[theorem]{Example}
\theoremstyle{remark}

\newcommand{\R}{\mathbb R}\newcommand{\C}{\mathbb C}\newcommand{\Q}{\mathbb Q}
\newcommand{\N}{\mathbb N}\newcommand{\A}{\mathbb A}
\newcommand{\E}{\mathbb E}\newcommand{\dR}{\mathrm{dR}}
\newcommand{\B}{\mathrm B}
\newcommand{\Cat}{\mathcal C}
\newcommand{\Ind}{\operatorname{Ind}}\newcommand{\Aut}{\operatorname{Aut}}
\newcommand{\Isom}{\operatorname{Isom}}\newcommand{\Span}{\operatorname{span}}
\newcommand{\per}{\operatorname{per}}\newcommand{\ev}{\operatorname{ev}}\newcommand{\Sym}{\operatorname{Sym}}

\newcommand{\norm}[1]{\left\lVert#1\right\rVert}
\newcommand{\pair}[2]{\langle #1,#2\rangle}
\newcommand{\mcP}{\mathcal P}\newcommand{\mcB}{\mathcal B}
\newcommand{\mcI}{\mathcal I}
\newcommand{\Hopf}{\mathcal H_{\dR}}

\begin{document}
\title[Sequential exponential-motivic periods]{Sequential exponential-motivic periods in spatially cut-off $\phi^4_2$}

\author[J. C. Sampedro]{Juan Carlos Sampedro} \thanks{This work has been supported by the Ministry of Science and Innovation of Spain under the
	Research Grants PID2025-171414NA-I00 and PID2024-155890NB-I00}
\address{Departamento de Matemática Aplicada y Ciencias de la Computación \\
	Avenida de los Castros 46 \\
	Universidad de Cantabria (UC) \\
	Santander, 39005, Spain.}
\email{juancarlos.sampedro@unican.es}

\begin{abstract}
We represent constructive integrals in spatially cut-off $\phi^4_2$ by sequences of exponential-motivic periods. For a nonzero spatial cutoff, a single Galois transformation makes the partition periods diverge in modulus along every sequence of approximations satisfying the constructive estimates. For a fixed cutoff presentation, we characterize the subgroup of transformations admitting a bounded realization on the free field's $L^1$ space and prove that all fixed-observable limits on this subgroup are a common scalar multiple of their physical values. A one-variable quartic example shows that all transformed fixed-observable sequences can converge without such a bounded realization, with limits that are not proportional to the physical values.
\end{abstract}
\smallskip

\maketitle

\setcounter{tocdepth}{1}
\tableofcontents

\section{Introduction}\label{sec:introduction}

Finite Gaussian integrals with polynomial interactions are periods of exponential motives. In constructive field theory, however, the fact that these integrals converge as the cutoff is removed does not identify the corresponding period symbols. The starting point of this paper is to keep those symbols rather than collapse them to their numerical values, and to ask what remains of their Galois structure in the constructive limit.

We work with the massive Euclidean free field on the plane, a fixed smooth nonnegative compactly supported spatial cutoff $g$, and a fixed positive quartic coupling. If $\mu_C$ is the free Gaussian measure and $V_g$ the Wick-renormalized quartic interaction, we consider
\begin{equation}\label{eq:target}
\mcI_g(q)=\int_{\mathcal S'(\mathbb R^2)}q(\phi)e^{-V_g(\phi)}\,d\mu_C(\phi),
\end{equation}
for fixed cylindrical polynomial observables $q$, that is, polynomials in finitely many smeared field variables. The value $\mcI_g(1)$ is the interacting partition function and is not used to normalize the functional. We use the spatially cut-off construction of Glimm and Jaffe \cite{GJ} and Simon \cite{Simon}, keeping both the spatial cutoff and the coupling fixed while removing the ultraviolet regularization and an auxiliary finite-dimensional discretization. The model requires Wick renormalization, while its ultraviolet limit is under constructive control. Keeping $g$ fixed also separates this problem from the additional infrared limit in which the spatial cutoff is removed.

Following \cite{Sampedro}, we smooth the field, replace the spatial integral by a finite Riemann sum, and refine these approximations together. This gives finite standard Gaussian spaces with polynomial interactions $f_r$ and representatives $q_r$ of $q$. The estimates in Section~\ref{sec:gaussian} allow us to choose the refinements so that $f_r\to V_g$ in $L^2(\mu_C)$ and $q_r e^{-f_r}\to qe^{-V_g}$ in $L^1(\mu_C)$. Hence
\begin{equation}\label{eq:introfinite}
\lim_{r\to+\infty}\int q_r e^{-f_r}\,d\gamma_r=\mcI_g(q),
\end{equation}
where $\gamma_r$ denotes the standard Gaussian measure at level $r$. Each integral on the left is the numerical value of an exponential-motivic period symbol $p_r^{\mathfrak m}(q)$, normalized by dividing out the free Gaussian integral. Rather than trying to identify these symbols at different cutoffs, we keep the sequence
\begin{equation}\label{eq:introsymbol}
\bigl(p_r^{\mathfrak m}(q)\bigr)_{r\ge1}
\end{equation}
modulo equality for all sufficiently large $r$, and call the resulting class a sequential exponential-motivic period. Fixing the observable before the cutoff varies prevents the sequential algebra from acquiring arbitrary cutoff-dependent scalars.

\begin{mainresult}[Sequential symbols and physical evaluation]\label{main:construction}
For the fixed rational cylindrical observables of Section~\ref{subsec:observables}, the finite symbols $p_r^{\mathfrak m}(q)$ generate a $\Q$-algebra of sequences modulo eventual equality. Numerical limiting evaluation defines an algebra homomorphism
\[
\bigl[p_r^{\mathfrak m}(q)\bigr]_{r\ge1}
\longmapsto \mcI_g(q)
\]
onto the algebra generated by the constructive integrals. In particular, a polynomial relation between fixed-observable symbols that holds at every sufficiently fine cutoff also holds between their limiting numerical values. Multiplication of symbols is realized by independent replicas of the field.
\end{mainresult}

To compare the Galois symmetries at different cutoffs, we place all finite motives in a common Tannakian category, so that one Galois group acts on their realizations compatibly with tensor products and motivic maps. Although this supplies a common group, it does not provide maps between successive interacting motives. The rule that keeps a polynomial insertion unchanged, for example, need not preserve the integration-by-parts relations when the interaction changes (Section~\ref{sec:transport}), which is why we keep sequences without choosing maps between the cutoffs.

With the same Galois transformation acting at every cutoff, we can ask whether the transformed numerical periods still converge. Starting from the fixed-observable sequences and all their Galois transforms, Section~\ref{sec:seqgalois} constructs the largest subalgebra on which numerical convergence is preserved by every transformation, but for the spatial model with $g\not\equiv0$, even the partition sequence lies outside this subalgebra.

To prove this failure, we multiply each total finite-dimensional action, including its Gaussian quadratic term, by a complex parameter and let the parameter go once around zero. Transporting the integration contours along this loop defines a single Galois transformation $\sigma_{\mathrm{rot}}$, independent of the cutoff. Denote by $Z_r^{\mathrm{rot}}$ the numerical value of the transformed partition symbol and by $D_r$ the number of points in the spatial Riemann sum at level $r$ where $g(x)>0$, which are precisely the sites contributing to the interaction $f_r$.

\Needspace{9\baselineskip}
\begin{mainresult}[Divergence of transformed partition periods]\label{main:divergence}
Suppose $g\not\equiv0$. Along every joint refinement satisfying the constructive estimates of Section~\ref{sec:gaussian},
\[
\left|Z_r^{\mathrm{rot}}\right|
\ge \exp\!\left(D_r-\sqrt{6D_r}\,\|f_r\|_{L^2(\mu_C)}\right)
\longrightarrow+\infty,
\]
although the original partition periods converge to $\mcI_g(1)$.
\end{mainresult}
The bound is proved in Theorem~\ref{thm:spatial-divergence}, and its divergence follows from the growth of $D_r$ and boundedness of $\|f_r\|_2$. Consequently, the Galois action at finite cutoff cannot restrict to the full algebra generated by the fixed-observable sequences, since every element of that algebra has convergent numerical periods. Section~\ref{sec:seqgalois} also treats the coaction, the map that records this Galois action algebraically, and explains why it cannot restrict to that algebra either.

For an individual Galois transformation, a sufficient condition for convergence is that a single bounded linear functional on $L^1(\mu_C)$ reproduce all its finite-cutoff periods. We call this condition \emph{physical $L^1$-continuity}, with the functional applied at each level to the polynomial insertion multiplied by the interaction weight. Such a functional must respect the integration-by-parts relations, since at a finite cutoff with total action $S$ the polynomials $\partial_iP-(\partial_iS)P$ integrate to zero against $e^{-S}dx$ and represent zero in the corresponding cohomology. We call these polynomials \emph{twisted-exact insertions} and denote by $\mathcal R_{\mathrm{phys}}$ their complex linear span after multiplication by the interaction weights and realization as functions on the common Gaussian field. Its elements have zero mean, and the next result shows that they are dense among all integrable functions of zero mean.

\begin{mainresult}[Saturation and physical $L^1$-continuity]\label{main:saturation}
The realized exact relations are dense in the entire mean-zero subspace, namely
\[
\overline{\mathcal R_{\mathrm{phys}}}^{\,L^1(\mu_C)}
=\left\{F\in L^1(\mu_C):\int F\,d\mu_C=0\right\}.
\]
Consequently, any continuous complex-linear functional $L$ on $L^1(\mu_C)$ that vanishes on these relations is a scalar multiple of integration, so
\[
L(F)=c\int F\,d\mu_C\qquad(F\in L^1(\mu_C)).
\]
\end{mainresult}

The transformations admitting a bounded realization form a subgroup, characterized in Corollary~\ref{cor:physical-subgroup} by preservation of the physical integration covectors up to a common scalar. It is the largest subgroup satisfying physical $L^1$-continuity for the full chosen presentation, and every fixed-observable limit on it equals $c\mcI_g(q)$, with the same nonzero scalar $c$ for all $q$.

The quartic model of Section~\ref{subsec:quarticprofile} shows that convergence alone need not force scalar limiting evaluations. Keeping the interaction confined to one variable and adjoining
independent Gaussian variables makes every fixed-observable symbol sequence eventually constant, so all its Galois translates have convergent numerical periods. An explicit contour rotation yields a limiting Galois evaluation whose values on two independent period symbols are not proportional to their physical values. By the preceding rigidity result, this evaluation admits no bounded realization on the physical $L^1$ space.

The distinction between period symbols and their numerical values follows the perspective emphasized by Cartier \cite[Secs.~8--9]{Cartier}. Related Galois structures appear in Connes--Marcolli's perturbative renormalization theory \cite{ConnesMarcolli} and Brown's motivic graph periods \cite{BrownCosmic}. Rosen's sequential periods \cite{Rosen} also retain arithmetic data along a sequence, although their index is crystalline Frobenius rather than an analytic cutoff. In the present setting the exponential interaction is kept at fixed coupling, so comparison with these constructions would require a relation between the corresponding symbols and Galois operations, not merely equality of numerical integrals.

For finite lattices, Weinzierl \cite{WeinzierlLattice} formulates integration-by-parts relations in twisted cohomology, and Gasparotto, Rapakoulias, and Weinzierl \cite{GRWLattice} derive differential equations in the action parameters. Our twisted-exact relations belong to the same finite-dimensional framework, but here they are used to test which transformed evaluations can extend continuously to the constructive physical space. We return to the comparison with perturbative and finite-lattice approaches in Section~\ref{subsec:cartier}.

\section{Gaussian approximation on \texorpdfstring{$\mathcal S'(\R^2)$}{tempered distributions}}\label{sec:gaussian}

In this section we approximate the constructive integrals by finite-dimensional Gaussian integrals. We first smooth the field and replace the spatial integral defining the interaction by a finite Riemann sum. We then prove convergence of the resulting integrals for each fixed polynomial observable as both approximations are refined.

We follow the construction of \cite[Sec.~5.2]{Sampedro}, starting from the constructive estimates recalled below. The Gaussian tools used in the argument are reviewed in Section~\ref{subsec:gaussian-toolbox}, following \cite[Chaps.~1--5 and~9]{Janson}, \cite{Simon}, and \cite[Sec.~5.1]{Sampedro}.

\subsection{The Gaussian field, Wiener chaos, and second quantization}\label{subsec:gaussian-toolbox}
Fix a mass $\mathfrak m>0$ and let $C=(-\Delta+\mathfrak m^2)^{-1}$. Let $H$ be the completion of $\mathcal S(\R^2;\R)$ for the covariance inner product
\[
\pair{u}{v}_H=\int_{\R^2}u(x)(Cv)(x)\,dx.
\]
On $\Omega=\mathcal S'(\R^2)$, endowed with the centered Gaussian measure $\mu_C$ of covariance $C$, put
\[
X(u)(\phi)=\phi(u),\qquad u\in\mathcal S(\R^2;\R).
\]
The map $u\mapsto X(u)$ extends by continuity to an isonormal Gaussian process on $H$, with
\[
\E[X(u)X(v)]=\pair{u}{v}_H,\qquad u,v\in H.
\]
For this Gaussian Hilbert-space construction see \cite[Chap.~1]{Janson}, and for the free measure and its marginals see \cite[Secs.~4--5]{Sampedro} and \cite{GJ,Simon}. If $(e_1,\ldots,e_d)$ is an orthonormal basis of a finite-dimensional subspace $K\subset H$, then
\[
Z_i=X(e_i),\qquad 1\le i\le d,
\]
are independent standard Gaussian variables. In these coordinates the marginal generated by $X(K)$ is
\[
d\gamma_d(z)=(2\pi)^{-d/2}e^{-|z|^2/2}\,dz
\]
on $\R^d$. Throughout, orthonormality refers to the covariance inner product on $H$, rather than the $L^2(\R^2)$ inner product.

\paragraph{Wick ordering and Wiener chaoses.}
Let $\mathcal P_{\le j}$ be the closed span in $L^2(\mu_C)$ of cylindrical polynomials of total degree at most $j$, and let $\Pi_{\le j}$ be the orthogonal projection onto it. Wick ordering removes the lower-degree component, so for $j\ge1$,
\[
:X(u_1)\cdots X(u_j):
=(\mathrm{id}-\Pi_{\le j-1})\bigl(X(u_1)\cdots X(u_j)\bigr).
\]
See \cite[Chap.~3]{Janson} and the Hermite-polynomial formulation in \cite[Prop.~4.1]{Nourdin}.

For a single Gaussian variable $G=X(u)$ of variance $c=\norm u_H^2$, write $:G^j:$ for the product with all $j$ factors equal to $G$, and put $:G^0:=1$. The first examples are
\[
:G:=G,\qquad :G^2:=G^2-c,\qquad :G^3:=G^3-3cG.
\]
In particular, the quartic expression used throughout the paper is
\begin{equation}\label{eq:wickquartic}
:X(u)^4:
=X(u)^4-6cX(u)^2+3c^2.
\end{equation}
These are the variance-rescaled Hermite polynomials. In the quartic case, subtracting the mean alone would not give orthogonality to the quadratic polynomials because the quadratic correction is essential.

The $j$th homogeneous Wiener chaos is the closed space generated by the Wick monomials of degree $j$. Equivalently, by polarization of the mixed products,
\[
\mathcal H_0=\C\,1,\qquad
\mathcal H_j=
\overline{\Span\{\,:X(u)^j:\mid u\in H\,\}}^{\,L^2(\mu_C)}
\quad(j\ge1).
\]
Orthogonality of distinct chaos orders follows from the covariance identity, valid for $u,v\in H$ and $j,\ell\ge1$,
\[
\E\bigl[:X(u)^j:\,:X(v)^\ell:\bigr]
=\begin{cases}
j!\pair{u}{v}_H^j,&j=\ell,\\
0,&j\ne\ell.
\end{cases}
\]
The Wiener--It\^o decomposition states that these spaces exhaust the Gaussian $L^2$ space, giving
\[
L^2(\mu_C)=\widehat{\bigoplus}_{j=0}^{\infty}\mathcal H_j.
\]
Both the closure and the orthogonal sum are taken in the $L^2$ norm, as in the decomposition of \cite[Chaps.~2--3]{Janson} and \cite[Sec.~4]{Nourdin}.

The Wick quartic belongs to $\mathcal H_4$ despite the lower-degree terms in its ordinary polynomial expression. Since $\mathcal H_4$ is closed, finite sums of Wick quartics and their $L^2$ limits also belong to it, as required for the spatially integrated interaction.

Write $H^{\odot j}$ for the symmetric part of $H^{\otimes j}$. The multiple Gaussian integral map $I_j:H^{\odot j}\to\mathcal H_j$ encodes the Wick polynomials by symmetric tensors, with
\begin{equation}\label{eq:wick-Ij}
:X(u)^j:=I_j(u^{\otimes j}).
\end{equation}
Polarization gives its action on finite symmetric tensors, and the Gaussian covariance identity above gives the continuous extension with normalization
\[
\norm{I_j(F)}_2^2=j!\norm F_{H^{\otimes j}}^2,
\qquad I_0(1)=1.
\]
This is the Wiener--It\^o realization \cite{Ito} in the Gaussian Hilbert-space formulation of \cite[Chaps.~4 and~7]{Janson}, with maps and tensor spaces complexified when needed. For $F\in\mathcal H_j$ and $p\ge2$, Gaussian hypercontractivity gives
\begin{equation}\label{eq:hypercontractive}
\norm F_p\le (p-1)^{j/2}\norm F_2.
\end{equation}
This is the fixed-chaos consequence of Nelson's theorem \cite{Nelson}, as presented in \cite[Chap.~5]{Janson} and \cite[Th.~4.5]{Nourdin}. Together with monotonicity of $L^p$ norms for $1\le p\le2$, this implies that $L^2$ convergence within a fixed chaos gives convergence in every finite $L^p$. The same formalism is used for $P(\phi)_2$ approximations in \cite[Secs.~5.1--5.2]{Sampedro}.

\paragraph{Second quantization.}
Let $T:H\to H$ be a contraction. Its Gaussian second quantization applies $T$ to each direction in a Wick product according to
\[
\Gamma(T):X(u_1)\cdots X(u_j):
=:X(Tu_1)\cdots X(Tu_j):,
\qquad \Gamma(T)1=1.
\]
Wick ordering uses the free Gaussian measure, with the covariance of the variables in each product. For example,
\[
\Gamma(T)\bigl(X(u)^2-\norm u_H^2\bigr)
=X(Tu)^2-\norm{Tu}_H^2.
\]
In the tensor notation this rule becomes
\begin{equation}\label{eq:secondquantization}
\Gamma(T)I_j(F)=I_j(T^{\odot j}F),\qquad F\in H^{\odot j}.
\end{equation}
Since $\norm{T^{\odot j}}\le1$, the norm identity for $I_j$ and the orthogonal chaos decomposition extend this rule to a contraction $\Gamma(T):L^2(\mu_C)\to L^2(\mu_C)$. This is the second-quantization formalism of \cite[Chaps.~4--5]{Janson} and \cite{Simon}.
If $T_n\to I_H$ strongly, then $T_n^{\odot j}\to I$ strongly on each fixed tensor power and hence
\begin{equation}\label{eq:Gamma-strong}
\Gamma(T_n)F \to F\quad\text{in }L^2(\mu_C),\qquad F\in\mathcal H_j.
\end{equation}
If $P_K$ is the orthogonal projection onto a finite-dimensional subspace $K\subset H$, then
\begin{equation}\label{eq:Gamma-conditioning}
\Gamma(P_K)F=\E[F\mid\sigma(X(K))],\qquad F\in L^2(\mu_C),
\end{equation}
the Gaussian conditional expectation identity of \cite[Chap.~9]{Janson}, \cite{Simon}, and \cite[Sec.~5.1]{Sampedro}.

For later use, suppose $T=T^*$ and $\norm T\le1$.  On a product probability space let $X'$ be an independent copy of $X$ and define
\[
Y(u)=X(Tu)+X'((I-T^2)^{1/2}u).
\]
Then $Y$ is again isonormal with the same law as $X$, and for suitable $F$ (in particular for the polynomial and exponential functions used below)
\begin{equation}\label{eq:Mehler-realization}
\Gamma(T)F(X)=\E_{X'}[F(Y)\mid X].
\end{equation}
This independent-noise representation is Mehler's formula for second quantization, whose scalar contraction $T=e^{-t}I_H$ is given in \cite[Th.~4.12]{Nourdin}. The conditional representation implies positivity, preservation of constants and integrals, and, by conditional Jensen (see \cite[Chap.~8]{Kallenberg}), the estimate
\begin{equation}\label{eq:Gamma-Jensen}
\exp\bigl(-s\Gamma(T)F\bigr)
\le \Gamma(T)(e^{-sF})
\quad(s\ge1)
\end{equation}
for real integrable $F$ with $e^{-sF}$ integrable. The right-hand side then also proves integrability of the left-hand side.

\subsection{Fixed rational observables}\label{subsec:observables}
Choose once and for all a countable family $(\xi_j)_{j\ge1}$ of real Schwartz functions whose linear span is dense in $H$.  Since the Gaussian sigma-algebra is generated by the isonormal coordinates of a dense set, the variables $X(\xi_j)$ generate the completed field sigma-algebra.  Define the algebra of actual random variables
\begin{equation}\label{eq:observables}
A_{\Q}=\Q[X(\xi_1),X(\xi_2),\ldots].
\end{equation}
An element $q\in A_{\Q}$ depends on finitely many prescribed coordinates and has rational coefficients in those variables. The covariance data and the coefficients obtained in finite orthonormal coordinates need not be algebraic over $\Q$.

Orthogonal Gaussian regression and Isserlis' moment formula \cite{Isserlis}, also called Wick's formula, show that conditioning preserves polynomial observables, as discussed in \cite[Chaps.~1, 3, and~9]{Janson}.

\begin{lemma}\label{lem:conditioning}
Let $K\subset H$ be finite-dimensional and let $q$ be a cylindrical polynomial of degree at most $d$.  Then
\[
q_K:=\E[q\mid\sigma(X(K))]
\]
is a polynomial of degree at most $d$ in any linear coordinate system on $K$.  Its coefficients are obtained from the coefficients of $q$ by field operations applied to the relevant covariance/Gram matrices.
\end{lemma}
\begin{proof}
Write the observable in the form $q=Q(X(u_1),\ldots,X(u_s))$ for finitely many $u_1,\ldots,u_s\in H$ and a polynomial $Q$ of total degree at most $d$. For each $j$ write
\[
u_j=P_Ku_j+(I-P_K)u_j.
\]
Writing $Y_j=X(P_Ku_j)$ and $Z_j=X((I-P_K)u_j)$ gives
$X(u_j)=Y_j+Z_j$, where the $Y$-variables are
$\sigma(X(K))$-measurable and the $Z$-variables are independent
of that sigma-algebra by Gaussian orthogonality of $K$ and
$K^\perp$.

Expand $Q(Y_1+Z_1,\ldots,Y_s+Z_s)$ as a finite sum of monomials $Y^\alpha Z^\beta$.  Conditional expectation fixes the $Y$-factor and replaces the $Z$-factor by the centered Gaussian moment $\E[Z^\beta]$, giving
\[
\E[Y^\alpha Z^\beta\mid\sigma(X(K))]
=Y^\alpha\E[Z^\beta].
\]
Odd total residual moments vanish, while Wick's formula expresses every even residual moment as a finite sum of products of covariances.  Hence the result is a polynomial in the variables $Y_j$ and has degree at most the original degree $d$.

Choose a basis $e_1,\ldots,e_m$ of $K$ and write
\[
P_Ku_j=\sum_{a=1}^m c_{ja}e_a.
\]
Then $Y_j=\sum_a c_{ja}X(e_a)$, so the conditional expectation is a polynomial in the linear coordinates $X(e_a)$.  In a nonorthonormal basis the coefficient vector $c_j=(c_{j1},\ldots,c_{jm})$ is obtained by solving
\[
G_K c_j=b_j,
\qquad (G_K)_{ab}=\langle e_a,e_b\rangle_H,
\qquad (b_j)_a=\langle u_j,e_a\rangle_H.
\]
Thus only inversion of the Gram matrix and Gaussian covariance moments are used, both of which are field operations in the covariance/Gram data specified in the statement, proving the polynomial and coefficient assertions.
\end{proof}

\subsection{The spatially cut-off constructive \texorpdfstring{$\phi^4_2$}{phi4(2)} interaction}\label{subsec:phi4-input}
Fix $g\in C_c^\infty(\R^2)$, $g\ge0$, and $\lambda>0$.
The compact support of $g$ is the fixed spatial cutoff.  Choose a real, even mollifier
\[
h\in C_c^\infty(\R^2),\qquad h\ge0,\qquad \int_{\R^2}h=1,
\]
and put
\[
h_n(y)=n^2h(ny),\qquad
\delta_{n,x}(y)=h_n(x-y)=n^2h(n(x-y)).
\]
The ultraviolet parameter is $n$. The random variable
\[
\phi_n(x):=X(\delta_{n,x})
\]
is a smoothed field value.  Translation invariance of $C$ implies that
\[
c_n:=\E[\phi_n(x)^2]=\norm{h_n}_H^2
\]
is independent of $x$.  The Wick-ordered quartic is therefore
\[
:\phi_n(x)^4:
=\phi_n(x)^4-6c_n\phi_n(x)^2+3c_n^2.
\]
Define the ultraviolet-regularized interaction as the $L^2(\mu_C)$ Bochner integral
\begin{equation}\label{eq:Vn}
V_n=\lambda\int_{\R^2}g(x):X(\delta_{n,x})^4:\,dx.
\end{equation}
Each $V_n$ lies in the fourth homogeneous chaos by \eqref{eq:wick-Ij}. For the Banach-space-valued integral used here and in the Riemann sums below, see the treatment of Bochner integration in \cite[Sec.~2.3]{Ryan}.

For a second, purely finite-dimensional approximation, let
\[
\Lambda_m=\{x\in2^{-m}\mathbb Z^2:|x|\le m\},
\]
and define
\begin{align}
H_{n,m}&=\Span_{\R}\{\delta_{n,x}:x\in\Lambda_m\},\label{eq:Hnm}\\
f_{n,m}&=\lambda2^{-2m}\sum_{x\in\Lambda_m}g(x):X(\delta_{n,x})^4:.
\label{eq:fnm}
\end{align}
For fixed $n$, the parameter $m$ refines the spatial Riemann sum.  Because $g$ has compact support, the expanding ball in $\Lambda_m$ eventually contains the support relevant to the sum, so this refinement does not involve an infinite-volume limit.  The random variable $f_{n,m}$ is measurable with respect to $\sigma(X(H_{n,m}))$ and lies in the fourth homogeneous chaos.

\paragraph{External constructive input.}
The only non-elementary constructive-QFT statement assumed in this section is the existence of a real random variable $V=V_g$ such that
\begin{equation}\label{eq:constructive-input}
V_n\to V\quad\text{in }L^2(\mu_C),
\qquad
e^{-V}\in L^p(\mu_C)\quad\text{for every }1\le p<\infty.
\end{equation}
See \cite[Prop.~8.5.1 and Th.~8.6.2]{GJ} or \cite[Ths.~V.3 and~V.7]{Simon} for this spatially cut-off $\phi^4_2$ result, and \cite[Sec.~5.2]{Sampedro} for its use in the finite-dimensional approximation.

The Riemann sums are those of \cite[Sec.~5.2 and Th.~5.2]{Sampedro}, but their exponentials require different estimates in the two limits. At fixed ultraviolet index, a lower bound on the Wick quartic suffices, whereas for the ultraviolet limit the identity $V_n=\Gamma(T_n)V$ in \eqref{eq:smoothedinteraction} lets us apply Mehler's formula and Jensen's inequality to obtain $\E[e^{-sV_n}]\le\E[e^{-sV}]$ and hence uniform integrability.

\begin{proposition}\label{hyp:constructive}
The limit $V$ belongs to the fourth homogeneous Wiener chaos.  For every $1\le p<\infty$,
\begin{align}
f_{n,m}&\longrightarrow V_n,&e^{-f_{n,m}}&\longrightarrow e^{-V_n}
&&\text{in }L^p(\mu_C)\quad(m\to\infty),\label{eq:innerconv}\\
V_n&\longrightarrow V,&e^{-V_n}&\longrightarrow e^{-V}
&&\text{in }L^p(\mu_C)\quad(n\to\infty).
\label{eq:outerconv}
\end{align}
For each fixed $n$, the increasing sigma-algebras $\sigma(X(H_{n,m}))$ generate the completed field sigma-algebra.
\end{proposition}
\begin{proof}
\medskip\noindent\textit{Step 1. Spatial approximation.}
For a fixed ultraviolet index $n$, continuity of translation in $H$
shows that $x\mapsto\delta_{n,x}$ is continuous and has constant norm
$\sqrt{c_n}$. The Wick representation \eqref{eq:wick-Ij} gives
\[
:X(\delta_{n,x})^4:=I_4(\delta_{n,x}^{\otimes4}).
\]
The map $u\mapsto I_4(u^{\otimes4})$ is continuous from $H$ to $L^2(\mu_C)$ because
\[
\|I_4(u^{\otimes4})-I_4(v^{\otimes4})\|_2^2
=4!\,\|u^{\otimes4}-v^{\otimes4}\|_{H^{\otimes4}}^2.
\]
Multiplication by the compactly supported function $g$ therefore gives a compactly supported continuous map
\[
x\longmapsto g(x):X(\delta_{n,x})^4:
\quad\text{from }\R^2\text{ to }L^2(\mu_C).
\]
Its Bochner integral is $V_n$, and its dyadic Bochner Riemann sums are exactly $f_{n,m}$ for all sufficiently large $m$.  Hence
$f_{n,m}\to V_n$ in $L^2(\mu_C)$.
All differences belong to the fourth homogeneous chaos, so hypercontractivity \eqref{eq:hypercontractive} gives, for every finite $p\ge2$,
\[
\|f_{n,m}-V_n\|_p
\le (p-1)^2\|f_{n,m}-V_n\|_2\to0.
\]
For $1\le p<2$ use monotonicity of $L^p$ norms on the probability space.  This proves the first convergence in \eqref{eq:innerconv} for every finite $p$.

For every real $t$,
\[
t^4-6c_nt^2+3c_n^2=(t^2-3c_n)^2-6c_n^2\ge-6c_n^2.
\]
Since $g\ge0$,
\[
f_{n,m}\ge -6\lambda c_n^2\,2^{-2m}\sum_{x\in\Lambda_m}g(x),
\qquad
V_n\ge -6\lambda c_n^2\int_{\R^2}g(x)\,dx.
\]
Since the scalar Riemann sums are bounded uniformly in $m$, there is a finite $B_n$ with
\[
f_{n,m}\ge-B_n,\qquad V_n\ge-B_n
\]
for all $m$.  On $[-B_n,\infty)$ the derivative of $e^{-t}$ has absolute value at most $e^{B_n}$. Therefore
\[
|e^{-f_{n,m}}-e^{-V_n}|
\le e^{B_n}|f_{n,m}-V_n|.
\]
Taking $L^p$ norms proves $e^{-f_{n,m}}\to e^{-V_n}$ in every finite $L^p(\mu_C)$.

\medskip\noindent\textit{Step 2. Ultraviolet convergence.}
The convergence $V_n\to V$ in $L^2$ supplied by
\eqref{eq:constructive-input}, together with closedness of the fourth
homogeneous chaos, implies $V\in\mathcal H_4$. To relate the regularized
interactions to this limit, consider the convolution operator
$T_nu=h_n*u$, whose Fourier multiplier is $\widehat h(\zeta/n)$.
This multiplier is real because $h$ is real and even, while
$h\ge0$ and $\int h=1$ give
\[
|\widehat h(\zeta/n)|\le \int h=1.
\]
Moreover, $\widehat h(\zeta/n)\to1$ pointwise. Dominated convergence in the weighted Fourier realization of $H$ yields $T_nu\to u$ in $H$
for every $u\in H$.  Thus $T_n$ is a self-adjoint contraction with $T_n\to I_H$ strongly.

We now prove the identity
\begin{equation}\label{eq:smoothedinteraction}
V_n=\Gamma(T_n)V.
\end{equation}
Keeping $n$ fixed, we apply $\Gamma(T_n)$ to $V_\ell$ and use its
action $T_n^{\odot4}$ on the fourth chaos to obtain
\[
\Gamma(T_n)V_\ell
=\lambda\int g(x):X(T_n\delta_{\ell,x})^4:\,dx.
\]
Convolution is commutative, and evenness of the mollifiers gives
\[
T_n\delta_{\ell,x}=T_\ell\delta_{n,x}.
\]
For fixed $n$, $T_\ell\delta_{n,x}\to\delta_{n,x}$ in $H$ as $\ell\to\infty$, and translation invariance makes the norm of the difference independent of $x$, giving uniform convergence on the compact support of $g$.  Continuity of $u\mapsto I_4(u^{\otimes4})$ and the Bochner integral then give $\Gamma(T_n)V_\ell\to V_n$ in $L^2$.
On the other hand, $\Gamma(T_n)$ is an $L^2$ contraction, so
\[
\|\Gamma(T_n)V_\ell-\Gamma(T_n)V\|_2
\le\|V_\ell-V\|_2\to0.
\]
The $L^2$ limit is unique, proving \eqref{eq:smoothedinteraction}.

Since $V\in\mathcal H_4$ and $T_n\to I_H$ strongly, \eqref{eq:Gamma-strong} gives $V_n=\Gamma(T_n)V\to V$ in $L^2$. Again hypercontractivity on the fixed fourth chaos upgrades this to every finite $L^p$.

For $s\ge1$, the Mehler representation \eqref{eq:Mehler-realization} and Jensen's inequality give the pointwise conditional estimate
\[
e^{-s\Gamma(T_n)V}
\le \Gamma(T_n)(e^{-sV}).
\]
Taking expectations and using preservation of integrals by $\Gamma(T_n)$ yields
\[
\E[e^{-sV_n}]
\le\E[e^{-sV}]<\infty.
\]
The convergence $V_n\to V$ in $L^2$ implies
$e^{-V_n}\to e^{-V}$ in probability, while, for any finite $p\ge1$,
choosing $s>p$ in the preceding estimate gives
\[
\sup_n\E\bigl[|e^{-V_n}|^s\bigr]<\infty,
\]
so the resulting uniform integrability of $|e^{-V_n}|^p$ allows us to
apply Vitali's theorem \cite[Chap.~5]{Kallenberg}, yielding
$e^{-V_n}\to e^{-V}$ in $L^p$ and completing the proof of
\eqref{eq:outerconv}.

\medskip\noindent\textit{Step 3. Generation of the field sigma-algebra.}
To prove the sigma-algebra assertion, we first establish density of
the span of all translates of $h_n$ for a fixed $n$. If $u\in H$ is
orthogonal to every translate $h_n(x-\cdot)$, the Fourier realization
of this orthogonality says that the inverse Fourier transform of
\[
F(\zeta)=\frac{\widehat u(\zeta)\overline{\widehat h_n(\zeta)}}{|\zeta|^2+\mathfrak m^2}
\]
vanishes identically.  The product belongs to $L^1(d\zeta)$ by Cauchy--Schwarz in the weighted Fourier model of $H$, so Fourier uniqueness (see \cite[Chap.~8]{Folland}) gives $F=0$ almost everywhere. Since $h_n$ is nonzero and compactly supported, $\widehat h_n$ is real analytic and not identically zero. Its zero set has Lebesgue measure zero by the real-analytic zero-set theorem, for which a short proof is given in \cite{Mityagin}.  Hence $\widehat u=0$ almost everywhere and $u=0$.  Thus the translates span a dense subspace of $H$.

The dyadic points used in the sets $\Lambda_m$ are dense in $\R^2$.  Translation is continuous in $H$, so the span of the dyadic translates occurring in $\bigcup_mH_{n,m}$ is already dense in the span of all translates.  Consequently
\[
\overline{\bigcup_mH_{n,m}}^{\,H}=H.
\]
For an isonormal Gaussian process, the completed sigma-algebra generated by $X(K)$ for a dense linear subspace $K\subset H$ is the full Gaussian sigma-algebra.  Therefore the increasing sigma-algebras $\sigma(X(H_{n,m}))$ generate the completed field sigma-algebra, as claimed.
\end{proof}

Put $w=e^{-V}$ and $Z=\E[w]$. By Proposition~\ref{hyp:constructive}, $0<Z<\infty$, so the finite measure $w\mu_C$ becomes a probability measure after division by $Z$.

The following theorem expresses the constructive integral of a fixed polynomial observable as an iterated limit of finite-dimensional Gaussian integrals. The spatial discretization is removed first, at fixed ultraviolet regularization, and the ultraviolet limit is taken afterwards. The proof combines Sampedro's conditional-observable approximation \cite[Th.~5.3]{Sampedro} with Proposition~\ref{hyp:constructive}.

\begin{theorem}[Fixed-observable realization]\label{thm:fixed}
Let $q\in A_{\Q}$ and $q_{n,m}=\E[q\mid\sigma(X(H_{n,m}))]$.
Then
\begin{align}
q_{n,m}e^{-f_{n,m}}&\longrightarrow qe^{-V_n}
&&\text{in }L^1\quad(m\to\infty),\label{eq:innerq}\\
qe^{-V_n}&\longrightarrow qe^{-V}
&&\text{in }L^1\quad(n\to\infty).
\label{eq:outerq}
\end{align}
Consequently,
\begin{equation}\label{eq:doubleperiod}
\mcI_g(q)=\lim_{n\to\infty}\lim_{m\to\infty}
\int_{H_{n,m}^*}q_{n,m}e^{-f_{n,m}}\,d\mu_{n,m}.
\end{equation}
Here $H_{n,m}^*$ denotes the finite-dimensional coordinate dual of $H_{n,m}$ and $\mu_{n,m}$ is the Gaussian marginal associated with $X(H_{n,m})$. After choosing an $H$-orthonormal basis both are identified with $\R^{\dim H_{n,m}}$ equipped with its standard Gaussian measure.
\end{theorem}
\begin{proof}
For $q\in A_{\Q}$ and a fixed $n$, Proposition~\ref{hyp:constructive}
shows that the increasing sigma-algebras
$\mathcal F_{n,m}:=\sigma(X(H_{n,m}))$ generate the completed field
sigma-algebra. Since the Gaussian polynomial $q$ belongs to
$L^2(\mu_C)$, the convergence theorem for conditional expectations
\cite[Chap.~9]{Kallenberg} gives
$q_{n,m}=\E[q\mid\mathcal F_{n,m}]\to q$ in $L^2(\mu_C)$.

Write
\[
q_{n,m}e^{-f_{n,m}}-qe^{-V_n}
=(q_{n,m}-q)e^{-f_{n,m}}+q(e^{-f_{n,m}}-e^{-V_n}).
\]
Taking $L^1$ norms and applying Cauchy--Schwarz gives
\[
\begin{aligned}
\|q_{n,m}e^{-f_{n,m}}-qe^{-V_n}\|_1
\le \|q_{n,m}-q\|_2\,\|e^{-f_{n,m}}\|_2+\|q\|_2\,\|e^{-f_{n,m}}-e^{-V_n}\|_2.
\end{aligned}
\]
The first product tends to zero because $q_{n,m}\to q$ in $L^2$
and $\|e^{-f_{n,m}}\|_2$ is bounded by
Proposition~\ref{hyp:constructive}, while the second tends to zero by
the same proposition, proving \eqref{eq:innerq}.

The observable $q$ no longer depends on $n$, so
\[
\|qe^{-V_n}-qe^{-V}\|_1
\le\|q\|_2\,\|e^{-V_n}-e^{-V}\|_2\to0
\]
by \eqref{eq:outerconv}.  This is \eqref{eq:outerq}.

The random variables $q_{n,m}$ and $f_{n,m}$ are $\mathcal F_{n,m}$-measurable.  If $\mu_{n,m}$ denotes the Gaussian marginal of $X(H_{n,m})$, then by the defining property of a pushforward measure,
\[
\E_{\mu_C}[q_{n,m}e^{-f_{n,m}}]
=\int_{H_{n,m}^*}q_{n,m}e^{-f_{n,m}}\,d\mu_{n,m}.
\]
Taking first $m\to\infty$ using \eqref{eq:innerq}, and then $n\to\infty$ using \eqref{eq:outerq}, gives
\[
\lim_{n\to\infty}\lim_{m\to\infty}
\int q_{n,m}e^{-f_{n,m}}\,d\mu_{n,m}
=\E[qe^{-V}]=\mcI_g(q).
\]
This is \eqref{eq:doubleperiod}.
\end{proof}

\subsection{A common increasing finite-dimensional realization}\label{subsec:diagonal}
We choose one diagonal in $(n,m)$ for all fixed observables and enlarge its coordinate spaces, without changing the interaction, until each observable is represented without conditioning. If $H_{n,m}\subset K\subset H$, a direction in $K\ominus H_{n,m}$ is a \emph{Gaussian spectator} when the interaction is independent of that coordinate and the free measure factors over the orthogonal splitting.

\begin{proposition}[Admissible diagonal]\label{prop:diagonal}
There exist integers $m(r)\to\infty$ and increasing finite-dimensional subspaces $K_r\subset H$ with dense union such that
\[
H_{r,m(r)}\subset K_r,
\qquad
\xi_1,\ldots,\xi_r\in K_r.
\]
Thus the diagonal uses ultraviolet index $n=r$ and spatial-discretization index $m=m(r)$.  Put
\[
f_r=f_{r,m(r)}
\]
and regard $f_r$ as a polynomial/random variable on $K_r$ by making it independent of the spectator directions.  With $w_r=e^{-f_r}$,
\begin{equation}\label{eq:diagonalconv}
f_r\to V,
\qquad
w_r\to w
\qquad\text{in every finite }L^p(\mu_C).
\end{equation}
Every fixed $q\in A_{\Q}$ is $\sigma(X(K_r))$-measurable for all sufficiently large $r$, hence is represented by the same random variable at those levels, and $qw_r\to qw$ in $L^1$.
\end{proposition}
\begin{proof}
For each integer $r\ge1$, use \eqref{eq:innerconv} with $n=r$.  Since convergence holds in every finite $L^j$, choose $m(r)\ge r$ so large that for every integer $1\le j\le r$,
\begin{equation}\label{eq:diagonal-accuracy}
\|f_{r,m(r)}-V_r\|_j
+\|e^{-f_{r,m(r)}}-e^{-V_r}\|_j<r^{-1}.
\end{equation}
This is a finite list of inequalities at stage $r$, so a single $m(r)$ can satisfy all of them.

Set $K_0=\{0\}$ and define
\begin{equation}\label{eq:diagonal-spaces}
K_r=K_{r-1}+H_{r,m(r)}
    +\operatorname{Span}_{\R}\{\xi_1,\ldots,\xi_r\}.
\end{equation}
The spaces $K_r$ are finite-dimensional and increasing, and their union contains every $\xi_j$, hence
\[
\overline{\bigcup_rK_r}^{\,H}=H
\]
because the linear span of the $\xi_j$ was chosen dense.

For $1\le p<\infty$, choose an integer $j\ge p$. Writing
$f_r=f_{r,m(r)}$, the estimate \eqref{eq:diagonal-accuracy}
and monotonicity of $L^p$ norms give, for every $r\ge j$,
\[
\|f_r-V_r\|_p\le\|f_r-V_r\|_j<r^{-1}.
\]
Hence
\[
\|f_r-V\|_p
\le\|f_r-V_r\|_p+\|V_r-V\|_p\to0
\]
by \eqref{eq:outerconv}.  The same argument gives
\[
\|w_r-w\|_p
\le\|e^{-f_r}-e^{-V_r}\|_p+\|e^{-V_r}-e^{-V}\|_p\to0.
\]
This proves \eqref{eq:diagonalconv}.

If $q\in A_{\Q}$ depends only on $X(\xi_1),\ldots,X(\xi_N)$, then $\xi_1,\ldots,\xi_N\in K_r$ for every $r\ge N$.  Therefore
\[
q=\E[q\mid\sigma(X(K_r))]
\]
as an actual random variable for all $r\ge N$.
The remaining convergence follows from $q\in L^2$ and
$w_r\to w$ in $L^2$ through
\[
\|qw_r-qw\|_1\le\|q\|_2\|w_r-w\|_2\to0.
\]
This concludes the proof.
\end{proof}
A choice satisfying Proposition~\ref{prop:diagonal} will be called an \emph{admissible diagonal}. 

\subsection{Finite Gaussian levels, orthonormal coordinates, and the coefficient field}\label{subsec:field}
We fix a countable family of levels, including all subsequent diagonals and spectator enlargements, and a common field for the coefficients of their actions, conditioned observables, and coordinate changes.

\paragraph{The family of finite levels.}
Let
\begin{equation}\label{eq:finite-directions}
\mathcal D_0=
\{\xi_j:j\ge1\}\ \cup\
\{\delta_{n,x}:n,m\ge1,\ x\in\Lambda_m\},
\end{equation}
and fix an enumeration of this set. Since each $\Lambda_m$ is finite, $\mathcal D_0$ is countable. Define
\begin{equation}\label{eq:finite-space-family}
\mathcal K=
\{\operatorname{Span}_{\R}F:F\subset\mathcal D_0\text{ is finite}\}.
\end{equation}
The family $\mathcal K$ is countable, contains every $H_{n,m}$, and is closed under finite sums. In particular, every $K_r$ constructed by \eqref{eq:diagonal-spaces} belongs to $\mathcal K$, for any choice of the integers $m(r)$.

Our countable indexing set is
\begin{equation}\label{eq:finite-level-family}
\mathcal T=
\{(n,m,K):n,m\ge1,\ K\in\mathcal K,\ H_{n,m}\subset K\}.
\end{equation}
From now on, admissible diagonals and spectator enlargements are taken within $\mathcal T$, as the construction \eqref{eq:diagonal-spaces} always permits.
For $a=(n,m,K)$ put $K_a=K$, $d_a=\dim K$, and $f_a=f_{n,m}$, regarded as a random variable on $X(K)$.  Thus the interaction is unchanged when the space is enlarged.  The original level $(n,m)$ means $(n,m,H_{n,m})$, and a diagonal level $r$ means $(r,m(r),K_r)$.  If two levels have spaces $K$ and $L$, both interactions can be represented on $K+L$ by two indices in $\mathcal T$.  Adjoining finitely many observable directions $\xi_j$ likewise gives another level in $\mathcal T$.

\paragraph{A common coefficient field.}
Set
\[
k_0=
\Q\Bigl(
\mathfrak m,\lambda,
\{g(x):x\in\textstyle\bigcup_m\Lambda_m\},
\{\pair{u}{v}_H:u,v\in\mathcal D_0\}
\Bigr)\subset\R.
\]
To include the normalizing coefficients in Gram--Schmidt, define recursively
\begin{equation}\label{eq:coefficient-field}
\begin{split}
k^{(0)}&=k_0,\\
k^{(\nu+1)}&=k^{(\nu)}\bigl(\sqrt{b}:b\in k^{(\nu)},\ b>0\bigr),
\qquad
k=\bigcup_{\nu\ge0}k^{(\nu)}\subset\R,
\end{split}
\end{equation}
where every square root is the positive real one.  Put $H_k=\operatorname{Span}_k\mathcal D_0\subset H$.  For each $K\in\mathcal K$, select a basis from $\mathcal D_0\cap K$ by retaining, in the fixed enumeration, each vector not in the span of those already retained, until $\dim K$ vectors have been selected.  Apply Gram--Schmidt to this ordered basis, and denote the resulting orthonormal basis by $\mathbf e_K$.  At level $a$ write
\[
\mathbf e_a=\mathbf e_{K_a}=(e_{a,1},\ldots,e_{a,d_a}),
\qquad z_i=X(e_{a,i}).
\]
The $z_i$ are independent standard Gaussian variables.  Write $f_a(z)$ for the interaction in these coordinates and, for $q\in A_{\Q}$, let
\[
q_a=\E[q\mid\sigma(X(K_a))],
\]
also expressed in these coordinates. By Lemma~\ref{lem:conditioning}, $q_a$ is a polynomial, and the next lemma places its coefficients and those of $f_a$ in the field just defined.

\begin{lemma}\label{lem:commonfield}
The field $k$ is countable, contains the positive square root of each of its positive elements, and has the following properties.
\begin{enumerate}
\item For every $K\in\mathcal K$, all vectors of $\mathbf e_K$ lie in $H_k$.  Every $v\in H_k\cap K$ has coordinates in $k$ in that basis, and the orthogonal projection $P_K$ maps $H_k$ into $H_k\cap K$.
\item If $K\subset L$ with $K,L\in\mathcal K$, there is an orthonormal basis of $L$ extending $\mathbf e_K$ whose vectors belong to $H_k$.  Its change-of-basis matrix with $\mathbf e_L$ lies in $O_{\dim L}(k)$.  Thus the splitting $L=K\oplus(L\ominus K)$ can be written over $k$.
\item For every $a\in\mathcal T$ and every $q\in A_{\Q}$,
\[
q_a(z),\ f_a(z)\in k[z_1,\ldots,z_{d_a}].
\]
\end{enumerate}
\end{lemma}
\begin{proof}
The field $k_0$ is countably generated over $\Q$, hence countable. Adjoining positive square roots preserves countability, so each $k^{(\nu)}$ and their union $k$ are countable. If $b>0$ belongs to $k$, then $b\in k^{(\nu)}$ for some $\nu$, and $\sqrt b\in k^{(\nu+1)}\subset k$.

Bilinearity and the definition of $k_0$ give
$\pair{u}{v}_H\in k$ for $u,v\in H_k$, so subtraction of
orthogonal projections in Gram--Schmidt preserves $H_k$.
Each nonzero residual vector $v$ has $\|v\|_H^2\in k_{>0}$,
and closure under positive square roots ensures that its normalization
also lies in $H_k$. Consequently $\mathbf e_K\subset H_k$,
the coordinates $\pair{v}{e_{K,i}}_H$ of every
$v\in H_k\cap K$ belong to $k$, and
\[
P_Kv=\sum_i\pair{v}{e_{K,i}}_H e_{K,i}\in H_k\cap K
\qquad(v\in H_k).
\]
For an inclusion $K\subset L$, start with $\mathbf e_K$, append a finite generating family for $L$ from $\mathcal D_0$, and perform Gram--Schmidt, discarding zero residuals.  This gives the required adapted basis in $H_k$.  Entries of the transition matrix are inner products of two vectors in $H_k$, hence lie in $k$. The matrix is orthogonal because both bases are orthonormal.

For the projection and conditional covariance in an arbitrary basis $u_1,\ldots,u_d\in H_k$ of $K$, set
\[
G=(\pair{u_i}{u_j}_H)_{i,j},
\qquad b(v)=(\pair{u_i}{v}_H)_{i=1}^d.
\]
The Gram matrix $G$ is invertible, and for $v,w\in H_k$,
\begin{equation}\label{eq:field-projection-covariance}
\begin{split}
P_Kv&=\sum_{i=1}^d\bigl(G^{-1}b(v)\bigr)_i u_i,\\
\pair{v-P_Kv}{w-P_Kw}_H
&=\pair{v}{w}_H-b(v)^{\mathsf T}G^{-1}b(w)\in k.
\end{split}
\end{equation}
Indeed, the first formula is the unique solution of the equations $\pair{v-P_Kv}{u_j}_H=0$, and the second follows by expanding the inner product and substituting the first.  All entries of $G$ and $b(v)$ lie in $k$, as do those of $G^{-1}$.  For $K=\{0\}$, the projection is zero and the residual covariance is simply $\pair{v}{w}_H$.

Fix $a\in\mathcal T$, put $K=K_a$, and write $q=Q(X(\xi_{j_1}),\ldots,X(\xi_{j_s}))$ with $Q\in\Q[t_1,\ldots,t_s]$.  In the Gaussian decomposition
\[
X(\xi_{j_\ell})=X(P_K\xi_{j_\ell})+\eta_\ell,
\qquad
\eta_\ell=X(\xi_{j_\ell}-P_K\xi_{j_\ell}),
\]
the vector $(\eta_1,\ldots,\eta_s)$ is independent of $X(K)$.  Each $X(P_K\xi_{j_\ell})$ is a linear polynomial with coefficients in $k$, and the covariance entries $\E[\eta_\ell\eta_t]$ belong to $k$ by \eqref{eq:field-projection-covariance}.  Expanding $Q$ and integrating out the $\eta_\ell$ therefore gives a polynomial over $k$. As in Lemma~\ref{lem:conditioning}, each Gaussian moment is a sum of products of these covariance entries.  This proves the assertion about $q_a$.

Finally, for $a=(n,m,K)$ and $x\in\Lambda_m$, put
\[
\ell_{a,x}(z)=\sum_{i=1}^{d_a}\pair{\delta_{n,x}}{e_{a,i}}_H z_i.
\]
Since $\delta_{n,x}\in K$, this polynomial represents $X(\delta_{n,x})$ exactly.  Its coefficients and $c_n=\|\delta_{n,x}\|_H^2$ lie in $k$.  The Wick expansion gives the explicit formula
\begin{equation}\label{eq:finite-interaction-coordinates}
f_a(z)=\lambda2^{-2m}\sum_{x\in\Lambda_m}g(x)
\bigl(\ell_{a,x}(z)^4-6c_n\ell_{a,x}(z)^2+3c_n^2\bigr).
\end{equation}
Every coefficient on the right belongs to $k$, proving the assertion about $f_a$.
\end{proof}

For example, if $q=X(\xi_j)^2$ and $c_i=\pair{\xi_j}{e_{a,i}}_H$, the conditional polynomial is
\[
q_a(z)=\left(\sum_{i=1}^{d_a}c_i z_i\right)^2
       +\|\xi_j\|_H^2-\sum_{i=1}^{d_a}c_i^2.
\]
The last two terms give the variance of the component orthogonal to $K_a$, which vanishes when $\xi_j\in K_a$. More generally, $q_a=q$ as a random variable whenever $q$ is already $\sigma(X(K_a))$-measurable.

The field $k$ depends on the chosen data because the values $g(x)$ and the covariance entries may be transcendental, and we do not claim that changing $k$ leaves the Tannakian groups unchanged. The coordinate assertions of the lemma apply to $H_k$. Below, the Betti realization is over $\Q$ and the de Rham realization is over this fixed $k$.

\paragraph{The finite Euclidean action.}
For $a\in\mathcal T$ define
\begin{equation}\label{eq:finiteaction}
S_a(z)=\mathfrak q_{d_a}(z)+f_a(z),
\qquad
A_a=k[z_1,\ldots,z_{d_a}].
\end{equation}
Here $\mathfrak q_d(z)=\frac12\sum_{i=1}^d z_i^2$ is a regular polynomial on $\A_k^d$, without complex conjugation, whose restriction to real points is $|z|^2/2$. We use \emph{action} in the Euclidean field-theory sense, as in the density
\begin{equation}\label{eq:action-density}
e^{-f_a(z)}\,d\gamma_{d_a}(z)
=(2\pi)^{-d_a/2}e^{-S_a(z)}\,dz.
\end{equation}
For the physical quartic approximations one has
\[
S_a(z)\ge\tfrac12|z|^2-B_a
\]
for some constant $B_a$.  Indeed, each Wick quartic satisfies
\[
t^4-6ct^2+3c^2=(t^2-3c)^2-6c^2\ge-6c^2,
\]
and $f_{n,m}$ is a finite sum of such terms with nonnegative coefficients.  Passing to spectator enlargements leaves the interaction unchanged and merely adds free quadratic coordinates, so the same form of coercive bound persists.

\section{Exponential motives and Galois profiles}\label{sec:motives}
We now associate exponential-motivic period symbols with the finite-dimensional Gaussian integrals constructed in Section~\ref{sec:gaussian}. Polynomial insertions define classes in twisted de Rham cohomology, and the real integration cycle supplies the corresponding Betti covector. Dividing by the free Gaussian period recovers the normalization of the original integrals.

We place these objects in a common Tannakian category and study the Galois action on their period symbols. We also examine which de Rham classes give zero symbols for the chosen integration cycle.

The construction uses the theory of exponential motives and the comparison with rapid-decay cohomology developed in \cite{FJ,HR,Hien}. Numbered references to Fres\'an--Jossen \cite{FJ} refer to the version of 11~April 2024.

\subsection{Exponential cohomology and comparison}\label{subsec:exponentialcohomology}
Let $X$ be a smooth variety over the field $k\subset\C$ fixed in Section~\ref{subsec:field}, and let
\[
S:X\longrightarrow\A^1_k
\]
be a regular function.  We call $(X,S)$ a \emph{variety with potential}.  The exponential integrals relevant for this paper have the form
\[
\int_\gamma e^{-S}\omega,
\]
where $\omega$ is an algebraic differential form and the (possibly noncompact) chain $\gamma$ approaches infinity only through regions in which $\Re S\to+\infty$.

\paragraph{Twisted de Rham realization.}
The potential $S$ defines the rank-one algebraic connection
\[
\nabla_S=d-dS\wedge
\]
on $\mathcal O_X$.  Since $d(dS)=0$, one has $\nabla_S^2=0$, and hence the twisted de Rham complex
\begin{equation}\label{eq:twistedDRcomplex}
(\Omega_X^\bullet,\nabla_S)
=
\bigl(\mathcal O_X\xrightarrow{\,d-dS\wedge\,}\Omega_X^1
\longrightarrow\cdots\longrightarrow\Omega_X^{\dim X}\bigr).
\end{equation}
We write
\[
H^j_{\dR}(X,S)=\mathbb H^j\bigl(X,(\Omega_X^\bullet,\nabla_S)\bigr).
\]
For affine $X$, the terms of \eqref{eq:twistedDRcomplex} are quasi-coherent, so this hypercohomology is computed by global algebraic forms. This is exactly the sign convention of \cite[Sec.~7.1.1 and Def.~7.1.2]{FJ}, with their potential $f$ equal to our $S$. Multiplication by $e^{-S}$ intertwines the twisted differential with the ordinary differential through
\[
d(e^{-S}\omega)=e^{-S}(d\omega-dS\wedge\omega).
\]
The analytic horizontal sections of $d-dS$ are locally generated by $e^S$, since $(d-dS)e^S=0$, whereas the integration factor $e^{-S}$ is horizontal for the dual connection. This distinction fixes the signs of the twisted derivatives below.

\paragraph{Rapid-decay Betti realization.}
Put $X^{\mathrm{an}}=X(\C)$ and, for $R\in\R$, let
\[
X_R^+(S)=\{x\in X^{\mathrm{an}}:\Re S(x)\ge R\}.
\]
In the elementary description sufficient for the present affine examples, rapid-decay homology is the projective limit of relative homology groups
\begin{equation}\label{eq:rdhomology}
H_j^{\mathrm{rd}}(X,S;\Q)
=\varprojlim_{R\to+\infty}H_j\bigl(X^{\mathrm{an}},X_R^+(S);\Q\bigr).
\end{equation}
For $R'\ge R$ one has $X_{R'}^+(S)\subseteq X_R^+(S)$, and the inclusion of pairs induces the transition map from the relative homology group at $R'$ to the one at $R$. Formula~\eqref{eq:rdhomology} is \cite[Def.~3.1.1]{FJ} with the closed subvariety $Y$ empty. By \cite[Sec.~3.1.2]{FJ}, the transition maps are isomorphisms for all sufficiently large $R$, and the resulting spaces are finite-dimensional. Thus a class is represented by compatible relative classes whose boundaries are pushed into regions where $e^{-S}$ is exponentially small. Dually, relative cohomology forms an eventually constant direct system whose limit, the cohomological Betti realization, is canonically
\[
H^j_{\mathrm{rd}}(X,S;\Q)=\operatorname{Hom}_\Q(H_j^{\mathrm{rd}}(X,S;\Q),\Q).
\]
A rapid-decay homology class is therefore a covector on the Betti realization, with the real integration cycle supplying the physical covector used in the framing.

Integration gives the period pairing
\begin{equation}\label{eq:expperiodpairing}
H_j^{\mathrm{rd}}(X,S;\Q)\times H^j_{\dR}(X,S)
\longrightarrow\C,
\qquad
([\gamma],[\omega])\longmapsto\int_\gamma e^{-S}\omega.
\end{equation}
The comparison theorem identifies the complexification of twisted de Rham cohomology with rapid-decay cohomology. For smooth affine varieties, the perfect integration pairing is due to Hien--Roucairol \cite[Th.~2.7]{HR}, whose convention $d+dg\wedge$ corresponds here to $g=-S$. Fres\'an--Jossen give the functorial comparison for varieties with a potential, including relative cohomology and compatibility with cup products, in \cite[Th.~7.6.1]{FJ}. See also \cite{Hien} for the general theory of periods of flat algebraic connections.

\paragraph{Motivic lift.}
Fres\'an--Jossen \cite{FJ} lift these realizations to a $\Q$-linear neutral Tannakian category $\mathcal M_{\exp}(k)$ of exponential motives, with fiber functors
\[
\omega_\B:\mathcal M_{\exp}(k)\to\operatorname{Vect}_{\Q},\qquad
\omega_{\dR}:\mathcal M_{\exp}(k)\to\operatorname{Vect}_k,
\]
and, for every object $M$, a comparison isomorphism
\[
c_M:\omega_{\dR}(M)\otimes_k\C\xrightarrow{\sim}
\omega_\B(M)\otimes_{\Q}\C.
\]
For a smooth pair $(X,S)$ the category contains cohomology objects, denoted here by $H^j_{\exp}(X,S)$, whose realizations are
\begin{equation}\label{eq:exprealizations}
\omega_{\dR}\bigl(H^j_{\exp}(X,S)\bigr)=H^j_{\dR}(X,S),
\qquad
\omega_\B\bigl(H^j_{\exp}(X,S)\bigr)=H^j_{\mathrm{rd}}(X,S;\Q),
\end{equation}
and whose comparison is characterized by the integration pairing \eqref{eq:expperiodpairing}. The category and its cohomology objects are defined in \cite[Def.~4.2.3]{FJ}, with tensor and neutral Tannakian structure given by \cite[Th.~4.4.1]{FJ}, while the de Rham fiber functor and tensor-compatible comparison are constructed in \cite[Def.~8.2.2 and Sec.~8.2.3]{FJ}. External products induce graded motivic K\"unneth isomorphisms compatible with both realizations and integration \cite[Prop.~3.6.3 and Th.~4.4.1]{FJ}. The subsequent arguments use these realizations, comparison, tensor products, duals, subquotients, and Tannakian duality.

The construction applies to our pairs $(\A_k^{d_a},S_a)$ over any subfield $k\subset\C$, without an isolated-critical-locus assumption. Coercivity is needed for the real integration cycle, rather than for the existence of the motive. After Gaussian normalization, Lemma~\ref{lem:top} describes the polynomial representatives and their exact relations.

\subsection{Normalized finite motives}\label{subsec:normalized}
For \eqref{eq:finiteaction}, set
\begin{equation}\label{eq:normalizedmotives}
Q_a=H^{d_a}_{\exp}(\A_k^{d_a},S_a),\quad
L_a=H^{d_a}_{\exp}(\A_k^{d_a},\mathfrak q_{d_a}),\quad
N_a=Q_a\otimes L_a^\vee.
\end{equation}
As checked below, the Gaussian motive $L_a$ has rank one, so its dual implements division by the free Gaussian integral without inverting the interacting partition function. Let $\eta_a$ be the volume-form de Rham frame of $L_a$. Evaluation on the oriented real Gaussian cycle defines a nonzero covector $\chi_a^{\mathrm G}:\omega_\B(L_a)\to\Q$. Its dual basis vector $b_a$ is uniquely characterized by
\[
b_a\in\omega_\B(L_a),\qquad \chi_a^{\mathrm G}(b_a)=1.
\]
Gaussian comparison gives $c_{L_a}(\eta_a)=(2\pi)^{d_a/2}b_a$. Define
\[
a_a(P)=[P\,dz_1\wedge\cdots\wedge dz_{d_a}]\otimes\eta_a^\vee.
\]
The oriented real cycle defines a rational rapid-decay homology class for $S_a$. Indeed, $S_a(z)\to+\infty$ on $\R^{d_a}$, and the compact semialgebraic truncations $\{z\in\R^{d_a}:S_a(z)\le R\}$, at regular values $R$, have boundary in the positive-potential region. Their oriented fundamental chains define the compatible relative classes of the real cycle. Every polynomial form is integrable on it, since $|P(z)e^{-S_a(z)}|\le e^{B_a}|P(z)|e^{-|z|^2/2}$. Write $\chi_a(u)=u([\R^{d_a}])$ for the resulting covector on $\omega_\B(Q_a)$. On $\omega_\B(N_a)=\omega_\B(Q_a)\otimes\omega_\B(L_a)^\vee$, define
\[
\beta_a=\chi_a\otimes b_a,\qquad
\beta_a(u\otimes\ell)=\chi_a(u)\,\ell(b_a)
\quad\bigl(u\in\omega_\B(Q_a),\ \ell\in\omega_\B(L_a)^\vee\bigr).
\]
Thus $\beta_a$ is a rational covector. The dual Gaussian comparison satisfies $c_{L_a^\vee}(\eta_a^\vee)=(2\pi)^{-d_a/2}b_a^\vee$, with $b_a^\vee(b_a)=1$. Writing $dz=dz_1\wedge\cdots\wedge dz_{d_a}$, tensor compatibility and evaluation on the real cycle in \eqref{eq:expperiodpairing} give
\begin{equation}\label{eq:finiteperiod}
\begin{aligned}
\beta_{a,\C}c_{N_a}a_a(P)
&=(2\pi)^{-d_a/2}\chi_{a,\C}\bigl(c_{Q_a}[P\,dz]\bigr)\\
&=(2\pi)^{-d_a/2}\int_{\R^{d_a}}P(z)e^{-S_a(z)}\,dz
=\int Pe^{-f_a}\,d\gamma_{d_a}.
\end{aligned}
\end{equation}
The identity extends to complex polynomials.

In one variable the Gaussian de Rham complex is
\[
k[z]\xrightarrow{\,P\mapsto(P'-zP)dz\,}k[z]dz.
\]
Its kernel is zero by polynomial degree. In its cokernel, $[z\,dz]=0$ and $[z^jdz]=(j-1)[z^{j-2}dz]$, so the cokernel is generated by $[dz]$. This generator is nonzero, since $P'-zP=1$ has no polynomial solution. Tensoring these complexes proves that the $d_a$-variable Gaussian has cohomology only in degree $d_a$, of dimension one. Faithfulness of realization therefore makes $L_a$ an invertible object of the Tannakian category. Integration gives its displayed comparison constant. This also shows that adjoining Gaussian variables introduces no additional degrees in the K\"unneth decomposition used below.

\begin{lemma}\label{lem:top}
The map $a_a:A_a\to\omega_{\dR}(N_a)$ is surjective and
\begin{equation}\label{eq:twistedkernel}
\ker a_a=\sum_{i=1}^{d_a}B_{a,i}A_a,
\qquad B_{a,i}P=\partial_iP-(\partial_iS_a)P.
\end{equation}
\end{lemma}
\begin{proof}
Let $d=d_a$, $A_a=k[z_1,\ldots,z_d]$, and write $\mathrm{vol}=dz_1\wedge\cdots\wedge dz_d$.

Because $\A_k^d$ is affine and the terms of the twisted de Rham complex are quasi-coherent, hypercohomology is computed by global polynomial differential forms \cite[Sec.~7.1.3]{FJ}. There are no forms of degree $d+1$, so every top-degree form $P\,\mathrm{vol}$ is automatically closed and every de Rham class in degree $d$ has a polynomial representative, proving surjectivity of the unnormalized map
\[
P\longmapsto[P\,\mathrm{vol}].
\]
Tensoring with the fixed nonzero vector $\eta_a^\vee$ therefore proves surjectivity of $a_a$.

Every polynomial $(d-1)$-form can be written uniquely as
\[
\omega=\sum_{i=1}^d(-1)^{i-1}P_i\,
 dz_1\wedge\cdots\wedge\widehat{dz_i}\wedge\cdots\wedge dz_d.
\]
A direct calculation gives
\[
(d-dS_a\wedge)\omega
=\sum_{i=1}^d\bigl(\partial_iP_i-(\partial_iS_a)P_i\bigr)\,\mathrm{vol}
=\sum_{i=1}^dB_{a,i}P_i\,\mathrm{vol}.
\]
Thus a polynomial $P$ represents the zero top de Rham class if and only if
\[
P=\sum_{i=1}^dB_{a,i}P_i
\]
for some polynomials $P_i$.  Equivalently,
\[
\ker a_a=\sum_{i=1}^{d_a}B_{a,i}A_a.
\]
Again tensoring with the one-dimensional Gaussian dual changes neither the equation nor the kernel.
\end{proof}

\paragraph{Adding Gaussian spectators.}
The spectator enlargements of Proposition~\ref{prop:diagonal} preserve the normalized motive and its frames before numerical evaluation.

Let $K\subset K'\subset H$ be finite-dimensional subspaces. Put $U=K'\ominus K$, $d=\dim K$ and $s=\dim U$, and choose orthonormal coordinates $z=(z_1,\ldots,z_d)$ on $K$ and $y=(y_1,\ldots,y_s)$ on $U$. We assume that the changes from the fixed coordinates to these adapted coordinates are defined over $k$, as ensured by Lemma~\ref{lem:commonfield}(2) for $K,K'\in\mathcal K$.

For an interaction $f\in k[z]$, the two actions are
\[
S_K(z)=\mathfrak q_d(z)+f(z),\qquad
S_{K'}(z,y)=S_K(z)+\mathfrak q_s(y).
\]
Write $N_K=Q_K\otimes L_K^\vee$ and $N_{K'}=Q_{K'}\otimes L_{K'}^\vee$ for their normalized motives, as in \eqref{eq:normalizedmotives}, and $a_K,a_{K'}$ for the corresponding polynomial-class maps. Averaging over the additional variables defines $\mathsf C:k[z,y]\to k[z]$ by
\begin{equation}\label{eq:spectator-averaging}
\mathsf C(P)(z)=\int_{\R^s}P(z,y)\,d\gamma_s(y).
\end{equation}
This polynomial represents $\E[P\mid\sigma(X(K))]$ under the free Gaussian measure. For example, if $d=s=1$, then $\mathsf C(z^2y^2+y^3)=z^2$.

\begin{proposition}\label{prop:spectators}
Under these assumptions, there is a canonical isomorphism $\iota:N_{K'}\xrightarrow{\sim}N_K$ in $\mathcal M_{\exp}(k)$ such that
\begin{equation}\label{eq:spectator-class-map}
\omega_{\dR}(\iota)\bigl(a_{K'}(P)\bigr)
=a_K\bigl(\mathsf C(P)\bigr)
\qquad\text{for every }P\in k[z,y].
\end{equation}
If $f$ is bounded below on $\R^d$, it also preserves the physical Betti covectors, with
\begin{equation}\label{eq:spectator-betti-map}
\beta_{K'}=\beta_K\circ\omega_\B(\iota).
\end{equation}
\end{proposition}
\begin{proof}
Let $L_U=H^s_{\exp}(\A_k^s,\mathfrak q_s)$. The Gaussian cohomology calculation above shows that $L_U$ is invertible and that the Gaussian factor has cohomology only in degree $s$. Exponential K\"unneth \cite[Prop.~3.6.3 and Th.~4.4.1]{FJ} therefore gives
\[
Q_{K'}\simeq Q_K\otimes L_U,
\qquad
L_{K'}\simeq L_K\otimes L_U.
\]
Define $\iota$ to be the composite
\[
\begin{aligned}
N_{K'}
\simeq(Q_K\otimes L_U)\otimes(L_K^\vee\otimes L_U^\vee)\simeq N_K\otimes(L_U\otimes L_U^\vee)
\xrightarrow{\ \mathrm{id}\otimes\ev_{L_U}\ }N_K.
\end{aligned}
\]
Here $\ev_{L_U}:L_U\otimes L_U^\vee\to\mathbf1$ is the evaluation isomorphism, so this composite realizes Gaussian cancellation as a morphism of motives. Naturality of external products and evaluation makes it compatible with the $k$-defined orthogonal changes of coordinates. For successive enlargements, the resulting maps compose by tensor associativity and evaluation of the successive Gaussian factors.

We next compute the de Rham map. In one Gaussian variable $y$, twisted-exact one-forms have the form $(Q'-yQ)\,dy$. Taking $Q=1$ gives $[y\,dy]=0$, and taking $Q=y^{j-1}$ for $j\ge2$ gives
\[
0=\bigl[((j-1)y^{j-2}-y^j)\,dy\bigr],
\qquad
[y^jdy]=(j-1)[y^{j-2}dy].
\]
Starting with $[dy]$, this is exactly the recurrence for the moments of a standard Gaussian variable $Y$. Hence
\[
[y^jdy]=\E[Y^j]\,[dy]\qquad(j\ge0).
\]
Write a general polynomial as $P(z,y)=\sum_\alpha P_\alpha(z)y^\alpha$. The coordinates $Y_1,\ldots,Y_s$ are independent, so
\[
\mathsf C(P)(z)
=\sum_\alpha P_\alpha(z)\prod_{i=1}^s\E[Y_i^{\alpha_i}].
\]
Each moment is zero or an integer, proving in particular that this is a polynomial over $k$. Applying the one-variable cohomological recurrence in every $y$-coordinate shows that the de Rham K\"unneth isomorphism sends
\begin{equation}\label{eq:spectator-dr-reduction}
[P(z,y)\,dz\wedge dy]
\longmapsto
\bigl[\mathsf C(P)(z)\,dz\bigr]\otimes[dy].
\end{equation}
Here $dz=dz_1\wedge\cdots\wedge dz_d$ and $dy=dy_1\wedge\cdots\wedge dy_s$. The polynomial on the right remains inside the cohomology class, so the formula does not require multiplying a cohomology class by an arbitrary polynomial.

The Gaussian volume frames satisfy $\eta_{K'}=\eta_K\otimes\eta_U$ under K\"unneth, with $\eta_U=[dy]$. Thus tensoring \eqref{eq:spectator-dr-reduction} with $\eta_{K'}^\vee$ and evaluating $\eta_U^\vee(\eta_U)=1$ gives exactly \eqref{eq:spectator-class-map}.

Finally, suppose $f\ge-B$. Then $S_K(z)\ge|z|^2/2-B$ and $S_{K'}(z,y)\ge(|z|^2+|y|^2)/2-B$, so both real cycles define the physical covectors as above. Let $\chi_K$ be the real-cycle covector on $\omega_\B(Q_K)$, and let $b_K$ and $b_U$ be the Gaussian Betti frames dual to the corresponding oriented real cycles. The product orientation gives
\[
\chi_{K'}=\chi_K\otimes b_U^\vee,
\qquad b_{K'}=b_K\otimes b_U.
\]
By definition $\beta_K=\chi_K\otimes b_K$ and $\beta_{K'}=\chi_{K'}\otimes b_{K'}$. Since $b_U^\vee(b_U)=1$, contraction of these frames gives \eqref{eq:spectator-betti-map}.
\end{proof}

Equations~\eqref{eq:spectator-class-map}--\eqref{eq:spectator-betti-map} imply the familiar integral identity
\[
\int_{\R^{d+s}}P(z,y)e^{-f(z)}\,d\gamma_d(z)d\gamma_s(y)
=\int_{\R^d}\mathsf C(P)(z)e^{-f(z)}\,d\gamma_d(z)
\]
when $f$ is bounded below. We will use the motivic identifications to pass between the spaces $H_{n,m}$ and their enlargements in Proposition~\ref{prop:diagonal}.

\subsection{The Tannakian category generated by the cutoff motives}
Inside $\mathcal M_{\exp}(k)$, we restrict to the tensor category generated by the chosen cutoff family $\mathcal T$ of Section~\ref{subsec:field} and the Gaussian line
\[
G_0=H^1_{\exp}(\A_k^1,z^2/2).
\]
We define
\begin{equation}\label{eq:category}
\Cat_{\mathcal T}=\langle G_0,N_a:a\in\mathcal T\rangle_\otimes
\subset \mathcal M_{\exp}(k).
\end{equation}
Here the brackets denote the smallest strictly full subcategory containing the displayed objects and closed under finite direct sums, tensor products, duals, subobjects, and quotients. We write $\Cat=\Cat_{\mathcal T}$. Compatibility of its Galois action with limiting evaluation is studied in Section~\ref{sec:seqgalois}.

For the terminology and reconstruction theorem for neutral Tannakian categories used here, see \cite[Sec.~2]{DM}. The Betti realization restricts to an exact faithful tensor functor
\[
\omega_\B:\Cat\longrightarrow \operatorname{Vect}_{\Q}.
\]
Since $\Cat$ inherits rigidity and $\operatorname{End}(\mathbf1)=\Q$ from the ambient category, it is neutral Tannakian over $\Q$. The de Rham realization is instead a $k$-valued fiber functor
\[
\omega_{\dR}:\Cat\longrightarrow \operatorname{Vect}_{k}.
\]
For a $\Q$-algebra $R$, a point of the Betti Tannakian group is, by definition, a tensor-compatible natural family of automorphisms
\[
g_M:\omega_\B(M)\otimes_{\Q}R\xrightarrow{\sim}
\omega_\B(M)\otimes_{\Q}R,
\qquad M\in\Cat.
\]
Accordingly we set
\[
G_\B=\Aut^\otimes(\omega_\B|_\Cat),
\qquad
G_{\dR}=\Aut^\otimes(\omega_{\dR}|_\Cat).
\]
These are affine group schemes over $\Q$ and $k$, respectively.

Enumerate the generating family, let $\Cat_{\le M}$ be the subcategory generated by the first $M$ objects, and write $G_{\B,\le M}$ for its Betti group. Restriction of tensor automorphisms gives an inverse system with the following properties, by \cite[Th.~2.11 and Props.~2.20(b), 2.21(a)]{DM}.

\begin{proposition}\label{prop:groups}
There are faithfully flat restriction morphisms
\[
G_{\B,\le M+1}\longrightarrow G_{\B,\le M},
\]
each $G_{\B,\le M}$ is an algebraic group of finite type over $\Q$, and
\[
G_\B\simeq\varprojlim_M G_{\B,\le M}
\]
as affine group schemes.  Every finite-dimensional representation furnished by an object of $\Cat$ therefore factors through one of these finite-type quotients.
\end{proposition}
\begin{proof}
The direct sum $E_M$ of the first $M$ generators is a tensor
generator of $\Cat_{\le M}$, since every object of that category
is a subquotient of a finite direct sum of tensor products of
$E_M$ and its dual. By \cite[Th.~2.11 and Prop.~2.20(b)]{DM},
$G_{\B,\le M}$ is therefore a closed algebraic subgroup of
$\operatorname{GL}(\omega_\B(E_M))$ and hence is of finite type.

Restriction to $\Cat_{\le M}$ gives the morphism from $G_{\B,\le M+1}$ to $G_{\B,\le M}$. The inclusion of categories is fully faithful, and every subobject of an object of $\Cat_{\le M}$ still belongs to $\Cat_{\le M}$. These are precisely the hypotheses of \cite[Prop.~2.21(a)]{DM}, so the restriction morphism is faithfully flat.

The increasing union $\bigcup_M\Cat_{\le M}$ is closed under all the operations defining $\Cat$, since each operation involves only finitely many objects. It therefore equals $\Cat$. For any $\Q$-algebra $R$, a compatible family of tensor automorphisms on the finite stages defines one on $\Cat$ by acting on each object in any stage containing it. Compatibility makes this independent of the stage. This identifies the functors of points and proves the inverse-limit assertion. The representation attached to an object factors through any stage containing that object.
\end{proof}

\subsection{The comparison torsor and framed motivic periods}\label{subsec:torsor}
A framed motivic period is a matrix coefficient on the torsor of tensor-compatible comparisons between the two realizations \cite{BrownPeriods}. Since the fiber functors have different ground fields, we first extend the Betti realization to $k$ and set
\begin{equation}\label{eq:torsor}
T_{\mathcal T}
=\Isom^\otimes_k(\omega_{\dR},\omega_\B\otimes_{\Q}k).
\end{equation}
For a $k$-algebra $R$, an $R$-point of $T_{\mathcal T}$ is a tensor isomorphism
\[
t_M:\omega_{\dR}(M)\otimes_kR
\xrightarrow{\sim}
\omega_\B(M)\otimes_{\Q}R
\]
natural in $M\in\Cat$.  Standard Tannakian theory makes $T_{\mathcal T}$ a torsor. Composition on the Betti side gives a left action of $(G_\B)_k$, and composition on the de Rham side gives a right action of $G_{\dR}$.  The exponential comparison isomorphisms of the preceding subsection assemble to a distinguished complex point
\[
c\in T_{\mathcal T}(\C).
\]
Conjugation by $c$ identifies $(G_{\dR})_\C$ with $(G_\B)_\C$, but need not identify the fiber functors over $k$. For the torsor over a general field see \cite[Sec.~8.2.3]{FJ}. We use its matrix coefficients over $k$, while the presentation in \cite[Sec.~8.3]{FJ} is written over $\Q$.

Define the algebra of motivic periods attached to the chosen cutoff family by
\[
\mcP^{\mathfrak m}_{\mathcal T}=\mathcal O(T_{\mathcal T}).
\]
It is naturally a $k$-algebra.  If $M\in\Cat$, $v\in\omega_{\dR}(M)$ and
$\beta\in\omega_\B(M)^\vee$, the framed matrix coefficient
\[
[M,\beta,v]\in\mcP^{\mathfrak m}_{\mathcal T}
\]
is the regular function
\[
[M,\beta,v](t)=\beta(t_Mv).
\]
It is $k$-linear in $v$, $\Q$-linear in $\beta$, functorial in $M$, and multiplication is induced by tensor product.  Evaluation at the comparison point gives the numerical period map
\begin{equation}\label{eq:periodmap}
\per:\mcP^{\mathfrak m}_{\mathcal T}\longrightarrow\C,
\qquad
\per[M,\beta,v]=\beta(c_Mv).
\end{equation}
These are the usual matrix-coefficient conventions for motivic periods, as in \cite[Sec.~2]{BrownPeriods}.

For the normalized motive $N_a$, the vector $a_a(q_a)$ and the physical rapid-decay covector $\beta_a$ constructed above therefore define
\begin{equation}\label{eq:finiteframed}
p_a^{\mathfrak m}(q)=[N_a,\beta_a,a_a(q_a)].
\end{equation}
By \eqref{eq:finiteperiod},
\[
\per\,p_a^{\mathfrak m}(q)
=\int q_a e^{-f_a}\,d\gamma_{d_a}.
\]
Proposition~\ref{prop:spectators} shows that enlarging the Gaussian space by spectator directions needed to represent $q$ gives the same framed symbol in the common category.

\subsection{Parity and the physical framing}\label{subsec:firstexample}
A nonzero de Rham class need not give a nonzero period symbol when paired with the physical Betti covector. The following quartic example illustrates how this can occur through symmetry.
\begin{example}\label{ex:parity}
Let $S(z)=z^2/2+\alpha z^4$ with $\alpha\in\Q_{>0}$, and use the normalized motive $N=Q\otimes L^\vee$ and frames $a(P),\beta$ of Section~\ref{subsec:normalized}. The observable $z$ has zero physical integral,
\[
\per[N,\beta,a(z)]
=(2\pi)^{-1/2}\int_{\R}z e^{-z^2/2-\alpha z^4}\,dz=0,
\]
because the integrand is odd. 

By Lemma~\ref{lem:top}, the class $[z\,dz]$ could vanish only if a polynomial $P$ satisfied
\[
P'-(z+4\alpha z^3)P=z.
\]
For nonzero $P$ of degree $d$, the left-hand side has degree $d+3$, since its highest-degree term comes from $-4\alpha z^3P$. It cannot equal $z$, and $P=0$ is also impossible. Thus $[z\,dz]\ne0$, and tensoring with the Gaussian dual frame gives $a(z)\ne0$.

To prove motivic vanishing, use the reflection $\rho(z)=-z$, which preserves both potentials. On polynomial forms it acts by
\[
\rho^*(P(z)\,dz)=-P(-z)\,dz.
\]
It acts by $-1$ on the Gaussian line $L$ and its dual. These signs cancel in $N=Q\otimes L^\vee$, so the induced involution $\tau$ satisfies
\[
\tau_{\dR}a(P)=a(P(-z)),
\qquad \tau_{\dR}=\omega_{\dR}(\tau).
\]
In particular, $\tau_{\dR}a(z)=-a(z)$. On the Betti side, reflection reverses the real cycle. Both its covector on $Q$ and the Gaussian Betti frame change sign, so their tensor product $\beta$ is fixed. Thus $\beta\circ\tau_\B=\beta$, where $\tau_\B=\omega_\B(\tau)$.

Every tensor comparison $t$ respects the motivic morphism $\tau$, giving $\tau_\B t_N=t_N\tau_{\dR}$. Consequently the regular function $p=[N,\beta,a(z)]$ satisfies
\[
p(t)=\beta(t_Na(z))
=\beta(\tau_\B t_Na(z))
=\beta(t_N\tau_{\dR}a(z))=-p(t).
\]
This identity holds functorially for every point of the comparison torsor, so $p=-p$ in the period algebra. In characteristic zero, $p=0$. The symmetry therefore forces vanishing before numerical evaluation, even though $a(z)\ne0$. This is a cohomological direction invisible to the physical framing and all its Galois translates, motivating the visible quotient below.
\end{example}

\subsection{Galois profiles and the visible quotient}
We now describe the kernel of the class-to-symbol map for a general physical framing.

Fix $a$.  Let $\Cat_a\subset\Cat$ be a finitely generated Tannakian subcategory containing $N_a$ (and $G_0$ when needed for the normalization), let
\[
G_{\B,a}=\Aut^\otimes(\omega_\B|_{\Cat_a}),
\qquad
T_a=\Isom^\otimes_k(\omega_{\dR}|_{\Cat_a},
\omega_\B|_{\Cat_a}\otimes_\Q k).
\]
Both are of finite type.  Restriction gives faithfully flat maps $G_\B\to G_{\B,a}$ and $T_{\mathcal T}\to T_a$, and every framed symbol involving $N_a$ is the pullback of its matrix coefficient on $T_a$.  Vanishing can therefore be tested on this finite-type torsor, without lifting complex points through the proalgebraic tower.

The comparison point restricts to $c_a\in T_a(\C)$.  Since $T_{a,\C}$ is a torsor under $(G_{\B,a})_\C$ and has the point $c_a$, it is trivial over $\C$.  With the convention that $G_{\B,a}$ acts on the Betti side, the map
\[
G_{\B,a}(\C)\longrightarrow T_a(\C),
\qquad g\longmapsto g^{-1}c_a,
\]
is a bijection.  Hence for $v\in\omega_{\dR}(N_a)$ we define the finite-cutoff Galois profile
\begin{equation}\label{eq:profile}
\Phi_{a,v}(g)=\beta_a g^{-1}c_{N_a}v,
\qquad g\in G_{\B,a}(\C).
\end{equation}

The common kernel of these translates is the realization of a maximal subobject $U_a$. The quotient $N_a^{\mathrm{vis}}$ is the smallest quotient through which the physical covector factors, as in \cite[Lem.~2.4, Sec.~2.4]{BrownPeriods}. We keep that covector fixed and allow all de Rham vectors, which differs from the minimal object attached to a single framed period.

\begin{theorem}\label{thm:profiles}
The following statements hold for $v\in\omega_{\dR}(N_a)$.
\begin{enumerate}
\item $\Phi_{a,v}(g)$ is the value of the motivic period $[N_a,\beta_a,v]$ at the comparison $g^{-1}c_a$.  Consequently
\[
[N_a,\beta_a,v]=0
\quad\Longleftrightarrow\quad
\Phi_{a,v}(g)=0\ \text{for every }g\in G_{\B,a}(\C).
\]
\item There is a largest subobject $U_a\subset N_a$ whose Betti realization is annihilated by $\beta_a$.  On complex fibers,
\[
\omega_\B(U_a)_\C
=\bigcap_{g\in G_{\B,a}(\C)}\ker(\beta_a g^{-1}).
\]
Writing $N_a^{\mathrm{vis}}=N_a/U_a$, one has
\[
[N_a,\beta_a,v]=0
\quad\Longleftrightarrow\quad
v\in\omega_{\dR}(U_a).
\]
\end{enumerate}
\end{theorem}
\begin{proof}
The comparison point $c_a\in T_a(\C)$ makes the $G_{\B,a}$-torsor $T_{a,\C}$ trivial. With our convention for the Betti action, every complex point of the torsor is uniquely of the form
\[
t=g^{-1}c_a,
\qquad g\in G_{\B,a}(\C).
\]
Hence the framed function $t\mapsto\beta_a(t_{N_a}v)$
becomes, in these coordinates,
\[
g\longmapsto\beta_ag^{-1}c_{N_a}v=\Phi_{a,v}(g).
\]

The group $G_{\B,a}$ is of finite type in characteristic zero. Hence $(G_{\B,a})_\C$ is reduced, and its complex points are Zariski dense. Since $T_{a,\C}\simeq(G_{\B,a})_\C$ after choosing $c_a$, a regular function on $T_a$ is zero if and only if its values at all complex points are zero. Thus the matrix coefficient on $T_a$ vanishes if and only if
\[
\Phi_{a,v}(g)=0\quad\text{for every }g\in G_{\B,a}(\C).
\]
The framed symbol in the ambient algebra $\mcP^{\mathfrak m}_{\mathcal T}$ is the pullback of this finite-level function along $T_{\mathcal T}\to T_a$. Because that morphism is faithfully flat, the induced pullback of coordinate rings is injective. This proves part~(1).

For part~(2), let $V_a=\omega_\B(N_a)$ and consider the $\Q$-linear matrix-coefficient map
\begin{equation}\label{eq:invisible-coefficient-map}
\Psi_a:V_a\longrightarrow \mathcal O(G_{\B,a}),
\qquad
\Psi_a(w)(g)=\beta_a(g^{-1}w).
\end{equation}
Here the right-hand side denotes the regular matrix coefficient of the finite-dimensional algebraic representation of $G_{\B,a}$ on $V_a$. Equivalently, $\Psi_a$ is obtained from the representation coaction by applying $\beta_a$ to the $V_a$-factor and composing with inversion on the group.  Define
\[
W_a:=\ker\Psi_a\subseteq V_a.
\]
Evaluation at the identity gives
\[
\Psi_a(w)(1)=\beta_a(w),
\]
so $W_a\subseteq\ker\beta_a$. To check stability, let $R$ be any $\Q$-algebra and let $h,g\in G_{\B,a}(R)$. After scalar extension to $R$,
\[
\Psi_{a,R}(hw)(g)
=\beta_{a,R}(g^{-1}hw)
=\Psi_{a,R}(w)(h^{-1}g).
\]
Thus left translation of the regular function $\Psi_a(w)$ describes the matrix coefficient of $hw$.  If $\Psi_a(w)=0$, all of its translations are zero, so the kernel is stable functorially in $R$. Equivalently, $W_a$ is an algebraic subrepresentation of $V_a$.

The subspace $W_a$ is maximal with these properties.  Indeed, if $U\subseteq V_a$ is any $G_{\B,a}$-stable subspace with $U\subseteq\ker\beta_a$, then the restriction of the representation to $U$ takes values in $U$ and $\beta_a|_U=0$.  Hence every matrix coefficient $g\mapsto\beta_a(g^{-1}u)$ with $u\in U$ is the zero regular function.  Therefore $\Psi_a(U)=0$ and $U\subseteq W_a$.

After scalar extension to $\C$, reducedness of $(G_{\B,a})_\C$ and Zariski density of its complex points give the pointwise formula
\begin{equation}\label{eq:invisible-orbit-kernel}
W_{a,\C}
=\bigcap_{g\in G_{\B,a}(\C)}\ker(\beta_ag^{-1}).
\end{equation}
Thus $W_a$ is precisely the largest $G_{\B,a}$-stable subspace of $V_a$ contained in $\ker\beta_a$.  By Tannakian duality it corresponds to a unique largest subobject $U_a\subset N_a$ whose Betti realization is annihilated by $\beta_a$.

Comparison sends
\[
\omega_{\dR}(U_a)_\C
\xrightarrow{\sim}
\omega_\B(U_a)_\C=W_{a,\C}.
\]
Thus for $v\in\omega_{\dR}(N_a)$,
\[
v\in\omega_{\dR}(U_a)
\iff c_{N_a}v\in W_{a,\C}
\iff \beta_ag^{-1}c_{N_a}v=0\quad\forall g\in G_{\B,a}(\C).
\]
By part~(1), the last condition is equivalent to $[N_a,\beta_a,v]=0$.
\end{proof}

\subsection{The finite-cutoff de Rham coaction}\label{subsec:finitecoaction}
The right action of the de Rham Tannakian group on the comparison torsor induces, contravariantly, a coaction on coordinate rings.  Put
\[
\mcP^{\dR}_{\mathcal T}=\mathcal O(G_{\dR}).
\]
Then
\[
\Delta:\mcP^{\mathfrak m}_{\mathcal T}
\longrightarrow
\mcP^{\mathfrak m}_{\mathcal T}\otimes_k\mcP^{\dR}_{\mathcal T}
\]
is the induced algebraic coaction. For a $k$-basis $(e_i)$ of $\omega_{\dR}(M)$ and its dual basis $(e_i^\vee)$, the matrix-coefficient formula of \cite[Sec.~2.2]{BrownPeriods} is
\begin{equation}\label{eq:coaction}
\Delta[M,\beta,v]
=\sum_i[M,\beta,e_i]\otimes[M,e_i^\vee,v]_{\dR}.
\end{equation}
Here $[M,e_i^\vee,v]_{\dR}$ is the regular function $g\mapsto e_i^\vee(g_Mv)$ on $G_{\dR}$. Basis independence follows from the canonical tensor $\sum_i e_i\otimes e_i^\vee$.

With $M=N_a$ and $\beta_a$ fixed, the symbols form a finite-dimensional comodule. Every de Rham vector has a polynomial representative by Lemma~\ref{lem:top}, so all first factors in \eqref{eq:coaction} are polynomial insertion symbols. To pass to sequences, we must control the coaction coefficients uniformly in $a$ and express those first factors using fixed observables (Section~\ref{sec:seqgalois}).

The profiles use the Betti group, whereas \eqref{eq:coaction} uses the de Rham group over $k$. Comparison relates them over $\C$, not in general over $\Q$. For either action to descend to numerical periods, the kernel of evaluation would have to be invariant.

\section{Sequential periods and constructive evaluation}\label{sec:sequences}
We now use the finite-cutoff symbols to define sequential periods and recover the constructive integrals by numerical limiting evaluation. Write
\[
P=\mcP^{\mathfrak m}_{\mathcal T}
\]
for the common period algebra. Here and in Section~\ref{sec:seqgalois}, the letter $P$ is reserved for this
algebra.

Two sequences of symbols represent the same sequential period if they agree at every sufficiently large cutoff. This retains relations between the symbols that cannot be recovered from their numerical limits alone. Along an admissible diagonal, constructive convergence defines an evaluation homomorphism on the algebra generated by the fixed-observable sequences. We then compare the relations along this diagonal with those of the iterated approximation.

\subsection{The ambient eventual-sequence algebra and fixed observables}
Fix the diagonal of Proposition~\ref{prop:diagonal}, and write
\[
q_r=\E[q\mid\sigma(X(K_r))],\qquad
p_r^{\mathfrak m}(q)=[N_r,\beta_r,a_r(q_r)]\in P,
\qquad P=\mcP^{\mathfrak m}_{\mathcal T}.
\]
Conditioning gives a canonical representative at every level, with $q_r=q$ for all sufficiently large $r$. The physical periods converge by construction.
\[
\per(p_r^{\mathfrak m}(q))=\E[q_rw_r]\longrightarrow\mcI_g(q).
\]
Define the \emph{ambient eventual-sequence algebra}
\begin{equation}\label{eq:ambientsequence}
\mathcal S(P)=\frac{\prod_rP}{\bigoplus_rP}.
\end{equation}
We shall also write $\mathcal S_0=\mathcal S(P)$ in Section~\ref{sec:seqgalois}. Define
\[
\mathbf p^{\mathfrak m}(q)=[(p_r^{\mathfrak m}(q))_r]\in\mathcal S_0.
\]

For a presentation by generators and relations, set
\[
R_{\Q}=\Sym_{\Q}(A_{\Q}),
\]
where $A_{\Q}$ is regarded as a rational vector space. Write $X_q$ for the image of $q$ under the canonical inclusion into $R_{\Q}$, a formal generator distinct from the Gaussian random variables $X(u)$. The inclusion is linear, so $X_{uq+vq'}=uX_q+vX_{q'}$, $u,v\in\Q$,
but it does not preserve multiplication. The generator $X_1$ represents the partition function, while the unit of $R_{\Q}$ represents the scalar $1$. Similarly, $X_{qq'}$ represents a single integral with insertion $qq'$, whereas $X_qX_{q'}$ represents the product of the integrals with insertions $q$ and $q'$.

The latter product can be realized using independent copies of the field. For $q_1,\ldots,q_s\in A_{\Q}$, let $\phi_1,\ldots,\phi_s$ be independent fields, each with law $\mu_C$. Since each $q_i e^{-V}$ belongs to $L^1(\mu_C)$, their product on this product probability space is absolutely integrable, with $L^1$ norm $\prod_i\|q_i e^{-V}\|_1$. Fubini's theorem gives
\begin{equation}\label{eq:replica}
	\prod_{i=1}^s\mcI_g(q_i)
	=\int \prod_{i=1}^s q_i(\phi_i)
	e^{-\sum_{i=1}^s V(\phi_i)}\,d\mu_C^{\otimes s}.
\end{equation}
The same factorization holds for the finite-cutoff integrals.

The linear assignment on observables extends uniquely to algebra homomorphisms
\[
\theta_r:R_{\Q}\to P,\qquad X_q\mapsto p_r^{\mathfrak m}(q).
\]
Define
\begin{equation}\label{eq:Jdiag}
J_{\mathrm{diag}}=\{F\in R_{\Q}:\theta_r(F)=0\text{ for all sufficiently large }r\},
\qquad \mcB_{\mathrm{diag}}=R_{\Q}/J_{\mathrm{diag}}.
\end{equation}
The set $J_{\mathrm{diag}}$ is an ideal because a common tail works for a finite sum of eventual relations and multiplication preserves zero at each cutoff. For example, for $q,h\in A_{\Q}$, the relation $X_q^2-X_h\in J_{\mathrm{diag}}$ means that
\[
p_r^{\mathfrak m}(q)^2=p_r^{\mathfrak m}(h)\quad\text{in }P
\]
for every sufficiently large $r$. The quotient therefore has the equivalent description
\begin{equation}\label{eq:fixedobservablealgebra}
\mcB_{\mathrm{diag}}
\simeq
\Q[\mathbf p^{\mathfrak m}(q):q\in A_{\Q}]
\subseteq \mathcal S(P),
\end{equation}
where the isomorphism sends $[X_q]$ to $\mathbf p^{\mathfrak m}(q)$. Its elements are finite polynomial expressions with rational coefficients in the fixed-observable sequences. Arbitrary scalar sequences in the ambient algebra are not included as additional generators.

\subsection{Renormalized evaluation as a quotient map}
An arbitrary element of $\mathcal S(P)$ need not have convergent numerical periods. On $R_{\Q}$ we can nevertheless define
\[
\ev_{\infty}:R_{\Q}\longrightarrow\R,
\qquad X_q\longmapsto\mcI_g(q),
\]
and set $J_{\mathrm{num}}=\ker\ev_{\infty}$. The inclusion $J_{\mathrm{diag}}\subseteq J_{\mathrm{num}}$ will allow this map to descend to the sequential quotient. We denote the induced map by $\ev_{\mathrm{ren}}$.

\begin{theorem}[Renormalized evaluation]\label{thm:seq}
There is a well-defined surjective $\Q$-algebra homomorphism
\begin{equation}\label{eq:reneval}
\ev_{\mathrm{ren}}:\mcB_{\mathrm{diag}}\longrightarrow
\mathcal A_g^{\mathrm{num}}:=\Q[\mcI_g(q):q\in A_{\Q}],
\qquad [X_q]\longmapsto\mcI_g(q).
\end{equation}
With the notation above, the sequence
\begin{equation}\label{eq:exactsequence}
0\longrightarrow J_{\mathrm{num}}/J_{\mathrm{diag}}
\longrightarrow\mcB_{\mathrm{diag}}
\xrightarrow{\ev_{\mathrm{ren}}}\mathcal A_g^{\mathrm{num}}
\longrightarrow0
\end{equation}
is exact.
\end{theorem}
\begin{proof}
Write $F\in R_{\Q}$ as $F=\mathcal P(X_{q_1},\ldots,X_{q_s})$, with fixed observables $q_i$ and $\mathcal P\in\Q[T_1,\ldots,T_s]$. Since $\theta_r$ and $\per$ are algebra homomorphisms, constructive convergence gives
\begin{equation}\label{eq:polynomial-period-limit}
\lim_r\per\theta_r(F)
=\mathcal P\bigl(\mcI_g(q_1),\ldots,\mcI_g(q_s)\bigr)
=\ev_{\infty}(F).
\end{equation}
If $F\in J_{\mathrm{diag}}$, the sequence on the left is eventually zero. Hence $J_{\mathrm{diag}}\subseteq J_{\mathrm{num}}$, and evaluation factors through $R_{\Q}/J_{\mathrm{diag}}$. Its image is $\mathcal A_g^{\mathrm{num}}$ by definition, and its kernel is $J_{\mathrm{num}}/J_{\mathrm{diag}}$.
\end{proof}
A relation in $J_{\mathrm{num}}/J_{\mathrm{diag}}$ may already hold numerically at every finite cutoff without holding motivically, or arise only after taking the limit. No injectivity of the finite numerical period map is assumed.

\subsection{Relations created by a limit}
The following quartic family gives a nonzero sequential class with zero limiting evaluation, using a fixed rational observable and a coupling tending to zero. It shows that norm convergence and fixed rational observables alone do not make \eqref{eq:reneval} injective, but does not prove non-injectivity for the spatial model, whose positive coupling in \eqref{eq:Vn} remains fixed.

\begin{example}[A fixed rational observable with a new limiting relation]\label{ex:emergent}
On a sequence of independent standard Gaussian coordinates consider
\[
S_r(x)=\tfrac12\sum_{j=1}^r x_j^2+r^{-1}x_1^4,
\qquad q=x_1^2-1.
\]
The observable has rational coefficients, degree two, and is independent of $r$. The weights $w_r=e^{-x_1^4/r}$ converge to $1$ in every finite $L^p$ by dominated convergence. Integrating the derivative of $xe^{-x^2/2-x^4/r}$ gives
\[
\E[(x_1^2-1)e^{-x_1^4/r}]
=-\frac4r\E[x_1^4e^{-x_1^4/r}]<0.
\]
Although the left-hand side tends to zero, its finite framed symbol cannot vanish because its numerical evaluation is nonzero at every cutoff. Hence $[X_q]\ne0$ in the diagonal sequence algebra, whereas its renormalized evaluation vanishes.
\end{example}

\subsection{Preserving iterated relations along a diagonal}

An admissible diagonal gives the right numerical limit, but may miss an eventual relation of the two-parameter approximation when that relation holds only beyond a threshold in $m$ depending on $n$. We can choose the diagonal to meet these thresholds as well as the analytic estimates. Countability of the observable algebra makes it possible to preserve all such relations simultaneously.

To specify the relations to be preserved, set
\[
\theta_{n,m}(X_q)
=[N_{n,m},\beta_{n,m},a_{n,m}(q_{n,m})]
\]
and extend to $R_{\Q}$ as before. Define
\begin{equation}\label{eq:Jit}
		J_{\mathrm{it}}=\{F\in R_{\Q}:
		\ \exists n_0\ \forall n\ge n_0\ \exists m_0(n) \ \forall m\ge m_0(n),\ \theta_{n,m}(F)=0\}.
\end{equation}
Thus $J_{\mathrm{it}}$ records eventual motivic relations in the original order of approximation, without choosing a diagonal. It is an ideal, and we write
\[
\mcB_{\mathrm{it}}=R_{\Q}/J_{\mathrm{it}}.
\]
This quotient uses the same observable generators as
$\mcB_{\mathrm{diag}}$, but imposes the relations of
\eqref{eq:Jit}. We call an analytically admissible diagonal
\emph{relation-preserving} if
$J_{\mathrm{it}}\subseteq J_{\mathrm{diag}}$.

The iterated convergence of Theorem~\ref{thm:fixed} gives an evaluation map on this quotient too.

\begin{proposition}\label{prop:iterevaluation}
	Renormalized evaluation also defines a surjection
	$\mcB_{\mathrm{it}}\to\mathcal A_g^{\mathrm{num}}$
	with kernel $J_{\mathrm{num}}/J_{\mathrm{it}}$.
\end{proposition}
\begin{proof}
For $F=\mathcal P(X_{q_1},\ldots,X_{q_s})$, Theorem~\ref{thm:fixed} and continuity of $\mathcal P$ give
\[
\lim_{n\to\infty}\lim_{m\to\infty}\per\theta_{n,m}(F)
=\ev_{\infty}(F).
\]
If $F\in J_{\mathrm{it}}$, each inner limit is zero for all sufficiently large $n$. Thus $J_{\mathrm{it}}\subseteq J_{\mathrm{num}}$, and the stated map and kernel follow by taking quotients.
\end{proof}
\begin{proposition}\label{prop:relationdiagonal}
There is an analytically admissible diagonal along which every relation in $J_{\mathrm{it}}$ holds eventually, so $J_{\mathrm{it}}\subseteq J_{\mathrm{diag}}$. It can be chosen with $m(r)$ above any prescribed integer sequence. For this choice there is a canonical surjection
\[
\mcB_{\mathrm{it}}\longrightarrow\mcB_{\mathrm{diag}}
\]
commuting with renormalized evaluation.
\end{proposition}
\begin{proof}
The algebra $A_{\Q}$ is countable because it consists of polynomials with rational coefficients in a countable set of generators.  Therefore the symmetric algebra $R_{\Q}$ is countable, and so is the subset $J_{\mathrm{it}}$.  Enumerate $J_{\mathrm{it}}=\{F_1,F_2,\ldots\}$, repeating elements if the set is finite.  For each $F_j$, choose an integer $n_j$ and, for every $n\ge n_j$, an integer threshold $M_j(n)$ such that
\begin{equation}\label{eq:iterated-relation-threshold}
m\ge M_j(n)\Longrightarrow\theta_{n,m}(F_j)=0.
\end{equation}
These data exist by the definition of $J_{\mathrm{it}}$.

At stage $r$, impose the simultaneous conditions
\[
\begin{cases}
m(r)\ge r,\\
m(r)\ge L(r),\\
m(r)\ge M_j(r)&\text{for every }j\le r\text{ with }n_j\le r,\\
\|f_{r,m(r)}-V_r\|_j+
\|e^{-f_{r,m(r)}}-e^{-V_r}\|_j<r^{-1}&1\le j\le r,
\end{cases}
\]
where $L(r)$ is any prescribed lower-bound sequence.  The list is finite for fixed $r$.  Convergence in $m$ from Proposition~\ref{hyp:constructive} allows one to choose a single $m(r)$ satisfying all of them.

Define $K_r$ by \eqref{eq:diagonal-spaces}.  Then $(r,m(r),K_r)\in\mathcal T$ for every $r$, so this extraction uses the family and coefficient field already fixed in Section~\ref{subsec:field}.  The analytic estimates above reproduce the proof of admissibility, so the resulting diagonal is analytically admissible and satisfies the prescribed lower bound.

For any $j$, once $r\ge\max(j,n_j)$ the stage-$r$ conditions
include $m(r)\ge M_j(r)$, so
\eqref{eq:iterated-relation-threshold} gives
$\theta_{r,m(r)}(F_j)=0$ on the original finite space
$H_{r,m(r)}$. Passing to $K_r$ only adds Gaussian spectators,
which Proposition~\ref{prop:spectators} cancels through the
canonical identification of normalized motives. For each fixed
observable $q$, the relation between the corresponding representatives
follows from the iterated Gaussian conditioning identity
\[
\E\bigl[\E[q\mid\sigma(X(K_r))]\mid\sigma(X(H_{r,m(r)}))\bigr]
=\E[q\mid\sigma(X(H_{r,m(r)}))].
\]
Under the spectator identification, these representatives determine
the same normalized de Rham class and framed symbol, giving
$\theta_r(F_j)=0$ for every $r\ge\max(j,n_j)$.
Each $F_j$ therefore belongs to $J_{\mathrm{diag}}$, which
proves $J_{\mathrm{it}}\subseteq J_{\mathrm{diag}}$.
Quotienting the identity map on $R_{\Q}$ gives the canonical surjection $\mcB_{\mathrm{it}}\twoheadrightarrow\mcB_{\mathrm{diag}}$,
and both renormalized evaluations send $X_q$ to the same number $\mcI_g(q)$, so the diagram commutes.
\end{proof}
This is an existence argument and gives no algorithm for finding the vanishing thresholds. For the chosen diagonal we have
\begin{equation}\label{eq:threeideals}
J_{\mathrm{it}}\subseteq J_{\mathrm{diag}}\subseteq J_{\mathrm{num}},
\end{equation}
and hence canonical surjections 
$\mcB_{\mathrm{it}}\twoheadrightarrow
\mcB_{\mathrm{diag}}\twoheadrightarrow
\mathcal A_g^{\mathrm{num}}$.

\section{Galois structures on sequential symbols}\label{sec:seqgalois}
We now ask how the Galois structures of the finite models pass to sequential symbols and interact with constructive evaluation. Applying a fixed Galois transformation at every cutoff always defines an action on sequences, but the transformed numerical periods need not converge. Passing the coaction itself to sequences imposes additional algebraic conditions.

We first examine these conditions and construct an algebra of sequential symbols stable under all Galois transformations. Within this algebra, we identify the largest invariant subalgebra on which limiting evaluation is defined. For a nonzero spatial cutoff, we then exhibit a single Galois transformation under which the partition periods diverge along every admissible diagonal. This excludes the partition symbol from that invariant subalgebra and rules out a coaction on the full fixed-observable algebra compatible with the finite coactions.

Throughout this section, we write
\[
P=\mcP^{\mathfrak m}_{\mathcal T},
\qquad
\Hopf=\mathcal O(G_{\dR}),
\qquad
\rho:P\longrightarrow P\otimes_k\Hopf
\]
for the common period algebra, the coordinate algebra of the de Rham
Galois group, and the finite coaction.

\subsection{The componentwise sequential map and the finite-span obstruction}
Retain $\mathcal S_0=\mathcal S(P)$ from Section~\ref{sec:sequences} and set
\begin{equation}\label{eq:seqtargets}
\mathcal S_1=
\frac{\prod_{r\ge1}(P\otimes_k\Hopf)}
{\bigoplus_{r\ge1}(P\otimes_k\Hopf)}.
\end{equation}
Here the tensor product is algebraic. Applying the finite coaction at each cutoff gives
\[
\widehat\rho:\mathcal S_0\longrightarrow\mathcal S_1,
\qquad [(p_r)]\longmapsto[(\rho(p_r))].
\]
This is a well-defined unital $k$-algebra homomorphism, since $\rho$ preserves eventual equality. Its target is larger than the space $\mathcal S_0\otimes_k\Hopf$ required for an ordinary coaction, and the following criterion describes the difference between these two spaces.

\begin{lemma}[The finite-span criterion]\label{lem:seqtensor}
The natural map
\[
\iota:\mathcal S_0\otimes_k\Hopf\longrightarrow\mathcal S_1,
\qquad [(p_r)]\otimes h\longmapsto[(p_r\otimes h)]
\]
is injective. A class $[(t_r)]$ belongs to its image if and only if there is a finite-dimensional subspace $W\subset \Hopf$ such that $t_r\in P\otimes_kW$ for all sufficiently large $r$.
\end{lemma}
\begin{proof}
Write an element of $\mathcal S_0\otimes_k\Hopf$ as a finite sum
\[
x=\sum_{i=1}^d[(p_{r,i})]\otimes h_i.
\]
After discarding redundant terms, assume $h_1,\ldots,h_d$ are $k$-linearly independent.  If $\iota(x)=0$, then the sequence
\[
\left(\sum_{i=1}^dp_{r,i}\otimes h_i\right)_r
\]
is finitely supported.  Thus for all sufficiently large $r$,
\[
\sum_{i=1}^dp_{r,i}\otimes h_i=0
\quad\text{in }P\otimes_k\Hopf.
\]
Linear independence of the $h_i$ implies $p_{r,i}=0$ for every $i$ on that tail, so each sequence $(p_{r,i})$ is finitely supported and represents zero in $\mathcal S_0$, giving $x=0$.

For the element above, each component of its image belongs to
\[
P\otimes_kW,
\qquad W=\operatorname{span}_k\{h_1,\ldots,h_d\}.
\]
Thus every element in the image satisfies the stated condition.

Conversely let $[(t_r)]\in\mathcal S_1$ and suppose there is a finite-dimensional $W\subset \Hopf$ and an index $r_0$ such that
\[
t_r\in P\otimes_kW\quad(r\ge r_0).
\]
Choosing a basis $h_1,\ldots,h_d$ of $W$, we can write each
$t_r$, $r\ge r_0$, uniquely in the form
\[
t_r=\sum_{i=1}^dp_{r,i}\otimes h_i.
\]
Completing these coefficient sequences by setting $p_{r,i}=0$
for $r<r_0$ gives
\[
\iota\left(\sum_{i=1}^d[(p_{r,i})]\otimes h_i\right)
=\left[\left(\sum_i p_{r,i}\otimes h_i\right)_r\right].
\]
The representative on the right agrees with $(t_r)$ for every $r\ge r_0$, so the two classes in $\mathcal S_1$ are equal.
\end{proof}
For example, let $P=\Hopf=k[t,t^{-1}]$ and $\rho(t)=t\otimes t$. The sequence $(t^r)_r$ has coaction $(t^r\otimes t^r)_r$, with just one term in each tensor but linearly independent second factors. These factors cannot eventually lie in one finite-dimensional subspace, whereas evaluation at $t=1$ gives the constant sequence $1$. Numerical convergence therefore does not imply the finite-span condition.

Even when the finite-span condition holds, a coaction on the algebra generated by fixed observables requires its first-factor sequences to remain in that algebra. To formulate this requirement over $k$, recall that $\mcB_{\mathrm{diag}}$ is only a $\Q$-subalgebra of $\mathcal S(P)$. Define its $k$-algebra span by
\[
B_k:=k[\mcB_{\mathrm{diag}}]\subseteq\mathcal S(P).
\]
Thus $B_k$ consists of finite $k$-linear combinations of elements of $\mcB_{\mathrm{diag}}$ and is the image of $k\otimes_{\Q}\mcB_{\mathrm{diag}}\to\mathcal S(P)$, without any assumption that this tensor map is injective.

An ordinary coaction on this algebra would have the form
\[
B_k\longrightarrow B_k\otimes_k\Hopf.
\]
Its first factors must be finite expressions in fixed-observable symbols with coefficients in $k$, a condition not ensured by arbitrary cutoff-dependent polynomial representatives. For the spatial model with $g\not\equiv0$, Theorem~\ref{thm:spatial-divergence} will show that this closure property and the finite-span condition cannot both hold on the full algebra.

\subsection{A stable envelope for the action by group points}
For a fixed group point, we can evaluate the second factor at every cutoff without a uniform finite span. Extend scalars at each level and put
\[
P_{\C}=P\otimes_k\C,\qquad
\mathcal S_{\C}=\frac{\prod_rP_{\C}}{\bigoplus_rP_{\C}},
\qquad G=G_{\dR}(\C).
\]
Here scalar extension is taken before the infinite product. Let $B_{\C}$ be the $\C$-subalgebra generated by the image of $\mcB_{\mathrm{diag}}$ in $\mathcal S_{\C}$. As with $B_k$, this is an image algebra, not an identification with $\mcB_{\mathrm{diag}}\otimes_{\Q}\C$.

For $\sigma\in G$, define $\sigma\cdot p=(1\otimes\ev_{\sigma})\rho_{\C}(p)$ and apply this operation componentwise to sequences. Our convention gives $\sigma\cdot(\tau\cdot p)=(\sigma\tau)\cdot p$. Each operation is an algebra automorphism, with inverse obtained from $\sigma^{-1}$, and it preserves eventual equality.

Define the algebra generated by all Galois translates of the fixed-observable algebra by
\begin{equation}\label{eq:galoisenvelope}
\mathcal A_{\mathrm{Gal}}
=\C[\sigma\cdot b:\sigma\in G,\ b\in B_{\C}]
\subseteq\mathcal S_{\C}.
\end{equation}
By the action law, this is the smallest $G$-stable subalgebra containing $B_{\C}$.

\subsection{The domain of limiting evaluation}\label{subsec:limiting-domain}
In the multiplicative example, $\sigma=2\in\C^\times$ sends $t^r$ to $2^rt^r$, so evaluation at $t=1$ changes the constant numerical sequence $1$ into the divergent sequence $2^r$. We need to distinguish convergence at the physical comparison from convergence after every transformation.

The comparison point extends the numerical period homomorphism to
\[
\per_{\C}:P_{\C}\to\C,\qquad p\otimes z\mapsto z\per(p).
\]
Define
\begin{equation}\label{eq:convdomain}
\mathcal A_{\mathrm{conv}}
=\{[(p_r)]\in\mathcal A_{\mathrm{Gal}}:
\lim_r\per_{\C}(p_r)\text{ exists in }\C\}.
\end{equation}
The map
\[
\ev_{\infty}:\mathcal A_{\mathrm{conv}}\to\C,
\qquad [(p_r)]\mapsto\lim_r\per_{\C}(p_r)
\]
is a well-defined unital algebra homomorphism. Theorem~\ref{thm:seq} implies $B_{\C}\subseteq\mathcal A_{\mathrm{conv}}$. On the image of $\mcB_{\mathrm{diag}}$, this limiting evaluation agrees with $\ev_{\mathrm{ren}}$, and composing with $R_{\Q}\to\mcB_{\mathrm{diag}}$ recovers the map called $\ev_{\infty}$ in Section~\ref{sec:sequences}.

\begin{proposition}[Maximal invariant evaluation domain]\label{prop:admissibleprofiles}
Set
\begin{equation}\label{eq:admdomain}
\mathcal A_{\mathrm{adm}}
=\bigcap_{\sigma\in G}\sigma^{-1}\cdot\mathcal A_{\mathrm{conv}}
=\{a\in\mathcal A_{\mathrm{Gal}}:\sigma\cdot a\in\mathcal A_{\mathrm{conv}}
\text{ for all }\sigma\in G\}.
\end{equation}
This is the largest $G$-stable subalgebra of $\mathcal A_{\mathrm{Gal}}$ contained in $\mathcal A_{\mathrm{conv}}$. The limiting profile
\begin{equation}\label{eq:limprofile}
\Phi_\infty:\mathcal A_{\mathrm{adm}}\to\operatorname{Fun}(G,\C),
\qquad \Phi_\infty(a)(\sigma)=\ev_{\infty}(\sigma\cdot a)
\end{equation}
is a unital algebra homomorphism and satisfies
\[
\Phi_\infty(\tau\cdot a)(\sigma)=\Phi_\infty(a)(\sigma\tau),
\qquad \Phi_\infty(a)(1)=\ev_{\infty}(a).
\]
Its kernel is a $G$-stable ideal, and the quotient by this ideal embeds equivariantly into the algebra of functions on $G$.
\end{proposition}
\begin{proof}
Each $\sigma^{-1}\cdot\mathcal A_{\mathrm{conv}}$ is a subalgebra. Their intersection is invariant because
\[
\sigma\cdot(\tau\cdot a)=(\sigma\tau)\cdot a.
\]
Any invariant subalgebra contained in $\mathcal A_{\mathrm{conv}}$ lies in every such inverse image, which proves maximality.

For each fixed $\sigma$, the map $a\mapsto\ev_{\infty}(\sigma\cdot a)$ is an algebra homomorphism on $\mathcal A_{\mathrm{adm}}$. Hence $\Phi_\infty$ is an algebra homomorphism to the pointwise function algebra. The action law also gives
\[
\Phi_\infty(\tau\cdot a)(\sigma)
=\ev_{\infty}((\sigma\tau)\cdot a)
=\Phi_\infty(a)(\sigma\tau).
\]
Evaluation at $1$ recovers $\ev_{\infty}$. The equivariance identity makes the kernel invariant, and quotienting by it gives the claimed embedding.
\end{proof}
We have the inclusions
\[
B_{\C}\subseteq\mathcal A_{\mathrm{conv}}\subseteq\mathcal A_{\mathrm{Gal}},
\qquad
\mathcal A_{\mathrm{adm}}\subseteq\mathcal A_{\mathrm{conv}}.
\]
These inclusions do not place $B_{\C}$ inside $\mathcal A_{\mathrm{adm}}$, since convergence at the identity is only one of the conditions in \eqref{eq:admdomain}.

Vanishing at the identity is weaker than having zero limiting profile. For an elementary illustration, use the multiplicative example with $B_{\C}$ generated by the constant sequence $a=[(t-1)_r]$. All its translates are constant sequences, and
\[
\Phi_\infty(a)(\sigma)=\sigma-1,\qquad
\Phi_\infty(a)(1)=0.
\]
Unlike the full profile kernel, the physical evaluation kernel is not stable because it contains $a$ but not its translate by $\sigma=2$.

Since the target in \eqref{eq:limprofile} is the algebra of all functions on $G$, pointwise convergence alone gives neither regularity of the profiles nor uniform convergence in $\sigma$.

For a fixed observable $q\in A_{\Q}$, the definition of $\mathcal A_{\mathrm{adm}}$ gives
\begin{equation}\label{eq:orbitadmissibility}
\mathbf p^{\mathfrak m}(q)\in\mathcal A_{\mathrm{adm}}
\quad\Longleftrightarrow\quad
\lim_r\per_{\C}(\sigma\cdot p_r^{\mathfrak m}(q))
\text{ exists for every }\sigma\in G.
\end{equation}
Membership is understood through the image in $\mathcal S_{\C}$. Since $B_{\C}$ is generated by these observable classes, \eqref{eq:orbitadmissibility} for all $q$ is equivalent to $B_{\C}\subseteq\mathcal A_{\mathrm{adm}}$. Although the constructive estimates provide the case $\sigma=1$, we now exhibit a single group point for which convergence fails already for $q=1$.

\subsection{A divergent Galois translate of the spatial partition sequence}\label{subsec:spatial-divergence}
We multiply the full potential, including its Gaussian quadratic term, by an auxiliary parameter $t$ and let $t$ travel once counterclockwise around zero, starting at $t=1$. For the quartic actions considered here, the real integration contour rotates to $-i\R^D$ in the interacting directions. Along the rotated contour, the quartic term is unchanged, while the quadratic term changes sign. We first show that this contour transport is induced by a single Galois element acting on all the finite models.

\begin{lemma}\label{lem:scaling-monodromy}
Let $k\subset\C$, and let $\gamma_{\mathrm{rot}}$ be the closed path of directions at infinity
\[
\gamma_{\mathrm{rot}}(\theta)=e^{-i\theta}\infty,
\qquad 0\le\theta\le2\pi,
\]
based at the positive real direction. Let $U_{\mathrm{rot}}$ denote forward parallel transport in rapid-decay cohomology along $\gamma_{\mathrm{rot}}$. Then
\[
U_{\mathrm{rot}}\in\Aut^\otimes(\omega_\B)
\]
is a tensor automorphism of the Betti fiber functor on $\mathcal M_{\exp}(k)$. For every exponential relative variety
$[X,Y,S,j,i]$, its component on
$H^j_{\mathrm{rd}}(X,Y,S;\Q)(i)$ is the forward monodromy obtained by following
\[
t=e^{i\theta}\longmapsto H^j_{\mathrm{rd}}(X,Y,tS;\Q)(i)
\]
once around $\C^\times$. In particular, the restriction of $U_{\mathrm{rot}}$ to
$\Cat_{\mathcal T}$ is a single tensor automorphism, independent of the cutoff level.
\end{lemma}
\begin{proof}
Fres\'an and Jossen construct in \cite[Sec.~11.12.2]{FJ}, for
$a\in k^\times$ and a path $\gamma$ from the direction $+\infty$ to
$a^{-1}\infty$ in the boundary of the real oriented blowup of
$\mathbb P^1(\C)$ at infinity, an automorphism of the category of
exponential motives compatible with tensor products, duals, and the
Betti realization. On the defining quiver it sends
$[X,Y,S,j,i]$ to $[X,Y,aS,j,i]$. The Betti identification in their
construction is written in the direction
\[
H^j_{\mathrm{rd}}(X,Y,aS;\Q)(i)
\xrightarrow{\ \sim\ }
H^j_{\mathrm{rd}}(X,Y,S;\Q)(i).
\]
Since $\gamma$ is directed from $+\infty$ to $a^{-1}\infty$,
the inverse of this Betti identification is the continuation map
in the forward direction of $\gamma$, from the rapid-decay
realization for $S$ to that for $aS$.

For $a=1$, the quiver endomorphism is the identity, so the inverse
of the Betti identification gives a natural automorphism of the
Betti representation. We may therefore regard this representation
as taking values in the abelian category of finite-dimensional
rational vector spaces equipped with an automorphism. The forgetful
functor to rational vector spaces is exact and faithful, and Nori's
universal property \cite[Th.~4.1.12]{FJ} extends this representation to an exact
functor from $\mathcal M_{\exp}(k)$ to this category whose underlying
vector-space functor is $\omega_\B$. The resulting
forward continuation is consequently natural for every motivic
morphism. The tensor and dual compatibility of the continuation in
\cite[Sec.~11.12.2]{FJ} is preserved on taking inverses, giving
\[
U_{\mathrm{rot}}\in\Aut^\otimes(\omega_\B).
\]

To identify the parameter convention, let $t=e^{i\theta}$ and
choose $R$ sufficiently large for the fixed pair and potential,
uniformly for $0\le\theta\le2\pi$. Rapid-decay cohomology for $tS$
is then computed by
\[
H^j\!\left(X^{\mathrm{an}},
Y^{\mathrm{an}}\cup S^{-1}(H_t(R));\Q\right),
\qquad
H_t(R)=\{u\in\C:\Re(tu)\ge R\}.
\]
Since $|t|=1$,
\[
H_t(R)=e^{-i\theta}H_1(R).
\]
Thus the decay direction follows $\gamma_{\mathrm{rot}}$ as
$\theta$ increases, identifying $U_{\mathrm{rot}}$ with the forward
monodromy of $H^j_{\mathrm{rd}}(X,Y,tS;\Q)(i)$ as $t$ travels once
counterclockwise around $0$.
\end{proof}

Restrict $U_{\mathrm{rot}}$ to $\Cat_{\mathcal T}$ and transport its inverse through the comparison isomorphism. With the action convention of Section~\ref{subsec:torsor}, set
\begin{equation}\label{eq:rotation-galois-point}
\sigma_{\mathrm{rot},M}=c_M^{-1}U_{\mathrm{rot},M}^{-1}c_M,
\qquad \sigma_{\mathrm{rot}}\in G_{\dR}(\C).
\end{equation}
Parallel transport preserves the pairing between homology and
cohomology. If $\beta'$ is the transported cycle covector, then
$\beta'(U_{\mathrm{rot},M}v)=\beta(v)$ for every Betti vector $v$,
so $\beta'=\beta\circ U_{\mathrm{rot},M}^{-1}$. This accounts for
the inverse in \eqref{eq:rotation-galois-point} and gives, for a
framed period,
\[
\per_{\C}\bigl(\sigma_{\mathrm{rot}}\cdot[M,\beta,v]\bigr)
=\beta(c_M\sigma_{\mathrm{rot},M}v)
=\beta(U_{\mathrm{rot},M}^{-1}c_Mv).
\]
The right-hand side is the integral over the transported cycle. The construction takes place on the ambient category before restriction to $\Cat_{\mathcal T}$, so the intermediate potentials $tS$ need not belong to the cutoff family.

\begin{theorem}[Divergence along every admissible diagonal]\label{thm:spatial-divergence}
Suppose $g\not\equiv0$. For an admissible diagonal $(r,m(r),K_r)$, put
\[
D_r=\#\{x\in\Lambda_{m(r)}:g(x)>0\}.
\]
As shown below, $D_r$ is the dimension of the interacting subspace, whereas the full space $K_r$ may also contain Gaussian spectators.
The point $\sigma_{\mathrm{rot}}$ of \eqref{eq:rotation-galois-point} satisfies
\begin{equation}\label{eq:spatial-divergence-bound}
\left|\per_{\C}(\sigma_{\mathrm{rot}}\cdot p_r^{\mathfrak m}(1))\right|
\ge \exp\!\left(D_r-\sqrt{6D_r}\,\|f_r\|_{L^2(\mu_C)}\right)
\longrightarrow+\infty.
\end{equation}
In particular, $\mathbf p^{\mathfrak m}(1)\notin\mathcal A_{\mathrm{adm}}$ and $B_{\C}\not\subseteq\mathcal A_{\mathrm{adm}}$ for every admissible diagonal.
\end{theorem}
\begin{proof}
\medskip\noindent\textit{Step 1. Reduction to the interacting subspace.}
At a finite level $a=(n,m,K)\in\mathcal T$, we isolate the sites
contributing to the interaction by writing
\[
I_m=\{x\in\Lambda_m:g(x)>0\},\qquad
D=|I_m|,\qquad b_x=\lambda2^{-2m}g(x),\qquad c=c_n.
\]
When $D=0$, the interaction vanishes and the Gaussian-normalized
partition symbol equals $1$, so the finite-cutoff bound holds with
equality. We may therefore assume $D>0$ and work on the interacting
subspace
\[
K^{\mathrm{act}}=\Span_{\R}\{\delta_{n,x}:x\in I_m\}.
\]
To establish linear independence of these translates, the Fourier
realization of the covariance norm turns any relation among them into
\[
\widehat h_n(\zeta)\sum_{x\in I_m}a_xe^{-ix\cdot\zeta}=0.
\]
Since $\widehat h_n(0)=1$, the exponential polynomial vanishes on a neighborhood of zero, hence identically, and all $a_x$ vanish. Thus $\dim K^{\mathrm{act}}=D$, with all directions in $K\ominus K^{\mathrm{act}}$ acting as Gaussian spectators. The splitting is defined over $k$ by the Gram--Schmidt construction of Section~\ref{subsec:field}, and Proposition~\ref{prop:spectators} cancels these factors motivically while preserving the partition framing. The reduced normalized motive is isomorphic to $N_a$ in the ambient category and hence belongs to the strictly full category $\Cat_{\mathcal T}$.

In orthonormal coordinates $z\in\R^D$ on $K^{\mathrm{act}}$, write $\ell_x(z)$ for $X(\delta_{n,x})$. The $D$ forms $\ell_x$ are linearly independent, satisfy $\E[\ell_x(Z)^2]=c$ for $Z\sim\gamma_D$, and give
\begin{equation}\label{eq:active-action}
S(z)=\frac{|z|^2}{2}+\sum_{x\in I_m}b_x
\bigl(\ell_x(z)^4-6c\ell_x(z)^2+3c^2\bigr)
=P_4(z)+P_2(z)+P_0.
\end{equation}
Here $P_j$ denotes the homogeneous term of degree $j$. Independence
of the forms $\ell_x$ and positivity of the coefficients $b_x$
make $P_4$ strictly positive on the unit sphere, so homogeneity gives
$P_4(z)\ge\kappa|z|^4$ for some $\kappa>0$ at this cutoff.

\medskip\noindent\textit{Step 2. Transport of the physical cycle.}
Along the loop $t=e^{i\theta}$, consider the contours
$\Gamma_\theta=e^{-i\theta/4}\R^D$, $0\le\theta\le2\pi$,
with the orientations induced by their parametrizations. The
transformed action satisfies
\[
e^{i\theta}S(e^{-i\theta/4}z)
=P_4(z)+e^{i\theta/2}P_2(z)+e^{i\theta}P_0.
\]
Its real part is bounded below by $\kappa|z|^4-C_2|z|^2-C_0$, with constants independent of $\theta$ at this cutoff. This bound keeps the tails of every contour in the rapid-decay region throughout the loop, so the family represents parallel transport of the real cycle. By Lemma~\ref{lem:scaling-monodromy}, this homological transport is dual to the inverse of the cohomological transport $U_{\mathrm{rot}}$ on $H^D_{\exp}(\A_k^D,S)$. The categorical construction requires no uniform estimate in $(n,m)$, since the displayed bound only verifies that the explicit contour represents the transported cycle at each cutoff. At the endpoint the contour is $-i\R^D$, with volume factor $(-i)^D$.

For the Gaussian normalization, the real contour moves through $e^{-i\theta/2}\R^D$ and returns with sign $(-1)^D$. The dual Gaussian factor has the reciprocal monodromy, giving the total phase
\[
\frac{(-i)^D}{(-1)^D}=i^D.
\]
Consequently the normalized period is
\begin{equation}\label{eq:rotated-partition}
\per_{\C}(\sigma_{\mathrm{rot}}\cdot p_a^{\mathfrak m}(1))=i^D J_{n,m},
\end{equation}
where
\begin{equation}\label{eq:positive-rotated-integral}
J_{n,m}=\frac1{(2\pi)^{D/2}}\int_{\R^D}
\exp\!\left(\frac{|z|^2}{2}
-\sum_{x\in I_m}b_x\bigl(\ell_x(z)^4+6c\ell_x(z)^2+3c^2\bigr)\right)dz.
\end{equation}
This integral is finite and positive. The spectator cancellation above is compatible with $U_{\mathrm{rot}}$, so \eqref{eq:rotated-partition} holds for every allowed enlargement $K$.

\medskip\noindent\textit{Step 3. Lower bound and divergence.}
To apply Jensen's inequality, we restore the standard Gaussian density
in \eqref{eq:positive-rotated-integral} and write, for $Z\sim\gamma_D$,
\[
J_{n,m}=\E\exp\!\left[|Z|^2-\sum_{x\in I_m}b_x
\bigl(\ell_x(Z)^4+6c\ell_x(Z)^2+3c^2\bigr)\right].
\]
The extra $|Z|^2/2$ compensates for the factor $e^{-|Z|^2/2}$ in $d\gamma_D$. Since $\E|Z|^2=D$, $\E\ell_x(Z)^2=c$ and $\E\ell_x(Z)^4=3c^2$, Jensen's inequality gives
\begin{align}
\log J_{n,m}
&\ge \E\!\left[|Z|^2-\sum_{x\in I_m}b_x
\bigl(\ell_x(Z)^4+6c\ell_x(Z)^2+3c^2\bigr)\right]\notag\\
&=D-12c^2\sum_{x\in I_m}b_x.
\label{eq:rotated-jensen}
\end{align}
We control the subtracted term by the $L^2$ norm of the interaction. The Wick covariance identity of Section~\ref{subsec:gaussian-toolbox} gives
\begin{align}
\|f_{n,m}\|_{L^2(\mu_C)}^2
&=24\sum_{x,y\in I_m}b_xb_y
\pair{\delta_{n,x}}{\delta_{n,y}}_H^4\notag\\
&\ge24c^4\sum_{x\in I_m}b_x^2
\ge\frac{24c^4}{D}\left(\sum_{x\in I_m}b_x\right)^2.
\label{eq:active-chaos-bound}
\end{align}
The first inequality keeps only the diagonal terms $x=y$, using $\|\delta_{n,x}\|_H^2=c$ and nonnegativity of every summand, while the second follows from Cauchy--Schwarz. It follows that
\[
12c^2\sum_{x\in I_m}b_x
\le\sqrt{6D}\,\|f_{n,m}\|_{L^2(\mu_C)}.
\]
Together with \eqref{eq:rotated-partition} and \eqref{eq:rotated-jensen}, this proves the finite-cutoff estimate in \eqref{eq:spatial-divergence-bound}.

Along an admissible diagonal, the interacting dimension $D_r$
tends to infinity because the refining dyadic grids contain
increasingly many points in an open ball on which $g$ is positive.
Meanwhile, \eqref{eq:diagonalconv} gives $f_r\to V$ in
$L^2(\mu_C)$, so $\|f_r\|_2\le M$ for some $M<\infty$.
Applying the finite-cutoff estimate therefore yields
\[
D_r-\sqrt{6D_r}\,\|f_r\|_2
\ge D_r-\sqrt6\,M\sqrt{D_r}\longrightarrow+\infty.
\]
The definition \eqref{eq:admdomain} now excludes $\mathbf p^{\mathfrak m}(1)$ from $\mathcal A_{\mathrm{adm}}$.
\end{proof}

The divergence follows from $D_r\to\infty$ and boundedness of $\|f_r\|_2$, so the same conclusion holds for the diagonals preserving iterated relations in Section~\ref{sec:sequences}. The same bound also gives divergence at fixed ultraviolet regularization. For fixed $n$, \eqref{eq:innerconv} implies that $\|f_{n,m}\|_{L^2(\mu_C)}$ remains bounded as $m\to\infty$, while the number of interacting directions tends to infinity when $g\not\equiv0$. Thus the transformed partition periods already diverge as the spatial discretization is refined, before taking the ultraviolet limit.

Theorem~\ref{thm:spatial-divergence} also rules out an ordinary coaction on $B_k$ compatible with the componentwise finite coactions. Such a coaction would make $B_{\C}$ stable under $G$. Since every element of $B_{\C}$ has convergent physical periods, this would imply $B_{\C}\subseteq\mathcal A_{\mathrm{adm}}$,
contradicting the theorem.

\section{Polynomial representatives and comparison maps}\label{sec:transport}
We now examine whether keeping the same polynomial representative defines a comparison map between the de Rham realizations of different cutoff models. Such a map is well defined only if every twisted-exact relation for the source interaction remains twisted-exact for the target interaction. We show that this condition fails whenever the interactions differ by a nonconstant polynomial, and apply the result to spatial refinement.

Work in a common finite Gaussian space $K$ of dimension $d$, with orthonormal basis $(e_1,\ldots,e_d)$, coordinates $z_i=X(e_i)$ and standard Gaussian measure $\gamma$. If the original spaces differ, Proposition~\ref{prop:spectators} permits their enlargement by free Gaussian spectators. For $f\in k[z]$, put $S_f=|z|^2/2+f$ and denote by $N_f$ the normalized motive of Section~\ref{subsec:normalized} for this action. Its polynomial-class map is
\[
a_f:k[z]\longrightarrow\omega_{\dR}(N_f),\qquad
a_f(Q)=[Q\,dz]_{S_f}\otimes\eta^\vee.
\]
Here $dz=dz_1\wedge\cdots\wedge dz_d$, the bracket denotes the class for the differential $d-dS_f\wedge$, and $\eta$ is the Gaussian volume-form frame used in that normalization.

For $h=\sum_i h_i e_i\in K$, define the directional derivative and its associated linear polynomial by
\[
D_hQ=\sum_i h_i\partial_{z_i}Q,\qquad
\ell_h(z)=\sum_i h_i z_i=X(h).
\]
The direction is defined over $k$ when every $h_i$ lies in $k$. For such a direction, $D_hS_f=\ell_h+D_hf$, and we define
\[
B_{f,h}Q=D_hQ-(\ell_h+D_hf)Q.
\]
By Lemma~\ref{lem:top}, $a_f(B_{f,h}Q)=0$ because this is the twisted derivative dictated by $\nabla_{S_f}=d-dS_f\wedge$.

The rule $a_f(Q)\mapsto a_{\widetilde f}(Q)$ descends to classes exactly when $\ker a_f\subseteq\ker a_{\widetilde f}$, and the following argument disproves this inclusion by exhibiting a polynomial that is exact for $f$ but has a positive period for $\widetilde f$.

\begin{proposition}\label{thm:stricttransport}
Let $f,\widetilde f\in k[z]$ be real polynomial interactions of degree at most four. Suppose both total actions are bounded below by $c|z|^2-C$ on $\R^d$, for some $c>0$ and $C\ge0$. Choose a direction $h$ defined over $k$ with $D_h(\widetilde f-f)\ne0$.

Then the rule $a_f(Q)\mapsto a_{\widetilde f}(Q)$ for $Q\in k[z]$ does not define a linear map between the two normalized de Rham spaces. More precisely, set $R=D_h(\widetilde f-f)$ and $P=B_{f,h}R$. This polynomial satisfies
\begin{equation}\label{eq:positivewitness}
\begin{gathered}
a_f(P)=0,\qquad a_{\widetilde f}(P)=a_{\widetilde f}(R^2)\ne0,\\
\per[N_{\widetilde f},\beta_{\widetilde f},a_{\widetilde f}(P)]
=\int R^2e^{-\widetilde f}\,d\gamma>0.
\end{gathered}
\end{equation}
The polynomial $P$ has degree at most six and is even when both interactions are even.
\end{proposition}
\begin{proof}
Subtracting the two twisted derivative operators gives the polynomial identity
\[
B_{f,h}Q=B_{\widetilde f,h}Q+(D_h\widetilde f-D_hf)Q.
\]
The first term on the right is exact for $\widetilde f$, so in its de Rham quotient
\begin{equation}\label{eq:defect}
a_{\widetilde f}(B_{f,h}Q)
=a_{\widetilde f}\bigl((D_h\widetilde f-D_hf)Q\bigr).
\end{equation}
With $Q=R=D_h(\widetilde f-f)$, the polynomial $P=B_{f,h}R$
is exact for $f$, whereas \eqref{eq:defect} gives
$a_{\widetilde f}(P)=a_{\widetilde f}(R^2)$. The finite period identity~\eqref{eq:finiteperiod} therefore evaluates this latter class as $\int R^2e^{-\widetilde f}\,d\gamma$. The quadratic lower bound on the total action makes the integral finite, while the nonzero real polynomial $R$ has a square that is positive on a nonempty open set. Since the weighted Gaussian density is positive everywhere, the integral is strictly positive and $a_{\widetilde f}(P)\ne0$. The proposed rule would send the zero vector $a_f(P)$ to this nonzero vector and therefore cannot define a linear map.

Finally $\deg R\le3$, while $B_{f,h}$ differentiates once and multiplies by polynomials of degree at most three. Thus $\deg P\le6$. If $f$ and $\widetilde f$ are even, $R$ is odd, so each of $D_hR$, $\ell_hR$, and $(D_hf)R$ is even. Their signed sum $P$ is even as well.
\end{proof}
Comparisons defined by other representatives or geometric correspondences remain possible.

\begin{corollary}
	\label{cor:spatialobstruction}
	Assume $g\not\equiv0$. For every fixed $n$ and all sufficiently
	large $m$, write $f_{n,m}$ and $f_{n,m+1}$ in common standard
	Gaussian coordinates $z$, adding independent Gaussian
	coordinates when needed. There exists an even polynomial $P\in k[z]$ of degree at most
	six such that
	\[
	a_{f_{n,m}}(P)=0,
	\qquad
	a_{f_{n,m+1}}(P)\ne0.
	\]
	Consequently, the assignment $a_{f_{n,m}}(Q)\mapsto a_{f_{n,m+1}}(Q)$ does not define a linear map between the normalized de Rham spaces.
\end{corollary}
\begin{proof}
The Fourier argument in Step~1 of the proof of
Theorem~\ref{thm:spatial-divergence} applies to any finite set of
distinct sites at fixed ultraviolet index $n$. Thus the vectors
$\delta_{n,x}$ are linearly independent, and the site variables
form nondegenerate linear coordinates on their span, although their
Gaussian covariance is not diagonal.

Since $g$ is positive on a nonempty open set $U$, density of
dyadic points provides $x\in U\cap\Lambda_M$ for some $M$.
Nestedness of the grids then ensures that $x\in\Lambda_m$
for all sufficiently large $m$.

Embed the levels $m$ and $m+1$ in a common finite Gaussian space.  In the corresponding site coordinate $t_x=X(\delta_{n,x})$, the coefficient of $t_x^4$ in the difference $F_m=f_{n,m+1}-f_{n,m}$ is
\[
\lambda g(x)\left(2^{-2(m+1)}-2^{-2m}\right)
=-\frac34\lambda2^{-2m}g(x)\ne0.
\]
Since no lower-degree Wick term can cancel this quartic coefficient, $F_m$ is a nonconstant polynomial.

All coefficients of $F_m$ lie in the fixed field $k$. Expressing $F_m$ in any coordinate basis defined over $k$, nonconstancy implies that at least one partial derivative is a nonzero polynomial. Choosing the corresponding coordinate direction $h$ gives $D_hF_m\ne0$. This direction is defined over $k$, and orthonormality is unnecessary for the algebraic obstruction.

For every Wick quartic,
\[
t^4-6ct^2+3c^2=(t^2-3c)^2-6c^2\ge-6c^2.
\]
Since all Riemann-sum coefficients $\lambda2^{-2m}g(x)$ are nonnegative, each finite total action in standard Gaussian coordinates satisfies
\[
S_{n,m}(z)=\frac12|z|^2+f_{n,m}(z)
\ge\frac12|z|^2-B_{n,m}
\]
for a finite constant $B_{n,m}$.  Spectator enlargement adds only terms $|y|^2/2$, so the same form of lower bound holds in the common comparison space.  All hypotheses of Proposition~\ref{thm:stricttransport}
are therefore satisfied. Applying it to $(f_{n,m},f_{n,m+1})$ gives the polynomial $P$ and the failure of the stated assignment.
\end{proof}

\section{Saturation, physical \texorpdfstring{$L^1$}{L1}-continuity, and Galois rigidity}\label{sec:rigidity}
Theorem~\ref{thm:spatial-divergence} leaves open which individual transformations have convergent fixed-observable periods. For a fixed $\sigma\in G_{\dR}(\C)$, we ask whether its transformed finite-cutoff periods define a single bounded functional of the realized $L^1(\mu_C)$ integrand, allowing cutoff-dependent insertions and cancellations between levels. This condition, called physical $L^1$-continuity, implies convergence for every fixed observable along an admissible diagonal.

\subsection{A finite-dimensional uniqueness theorem}
For the density $e^{-S}dx$, the score operators of Stein's method are $\partial_i-\partial_iS$. We adapt the multivariate density characterization of Mijoule, Rai\v{c}, Reinert, and Swan \cite[Prop.~3.14 and Def.~3.17]{Stein} to complex $L^p$-bounded functionals tested only on polynomials. Gaussian tails let us pass from those tests to a distributional equation.

\begin{theorem}[Continuous twisted functionals]\label{thm:finiterigidity}
Let $S\in\R[x_1,\ldots,x_d]$ satisfy $S(x)\ge a|x|^2-b$, $a>0$, and put
\[
Z_S=\int_{\R^d}e^{-S(x)}\,dx,\qquad d\nu=Z_S^{-1}e^{-S}dx.
\]
Suppose $1\le p<\infty$ and a complex-linear functional $\ell$ on $\C[x_1,\ldots,x_d]$ satisfies
\[
|\ell(P)|\le C\norm P_{L^p(\nu)},\qquad
\ell(\partial_iP-(\partial_iS)P)=0
\]
for all polynomials $P$ and all $i$. Then there is $c\in\C$ such that
$\ell(P)=c\int P\,d\nu$ for every $P$. Moreover, $\ell$ has a unique continuous extension to $L^p(\nu)$, given by
\[
\widetilde\ell(F)=c\int F\,d\nu\qquad(F\in L^p(\nu)).
\]
\end{theorem}
\begin{proof}
\medskip\noindent\textit{Step 1. Representation and integrability.}
The bound
\[
|\ell(P)|\le C\|P\|_{L^p(\nu)}
\]
allows Hahn--Banach to extend $\ell$ to $L^p(\nu)$ with the same
norm bound. By $L^p$ duality, any such extension
$\widetilde\ell$ is represented by a function $u\in L^{p'}(\nu)$,
where $p'=\infty$ when $p=1$, through
\[
\widetilde\ell(F)=\int uF\,d\nu
\]
for every $F\in L^p(\nu)$, with $\widetilde\ell$ agreeing with
$\ell$ on polynomials.

The Gaussian tails supplied by $S(x)\ge a|x|^2-b$ imply that,
for every polynomial $Q$ and every $R>0$,
\begin{equation}\label{eq:gaussian-exp-moment}
\int |Q(x)|^p e^{pR|x|}\,d\nu(x)<\infty.
\end{equation}
If $p>1$, H\"older gives
\[
\int|uQ|e^{R|x|}\,d\nu
\le\|u\|_{p'}
\left(\int|Q|^pe^{pR|x|}\,d\nu\right)^{1/p}<\infty.
\]
For $p=1$, use $\|u\|_\infty$ instead.  We shall use the resulting estimate in the form
\begin{equation}\label{eq:expdom}
\int |uQ|e^{R|x|}\,d\nu<\infty
\end{equation}
for every polynomial $Q$ and every $R>0$.

\medskip\noindent\textit{Step 2. The distributional equation.}
To pass from polynomial tests to Fourier tests, we approximate
$e^{it\cdot x}$, for a fixed $t\in\R^d$, by its Taylor polynomials
\[
P_N(x)=\sum_{j=0}^N\frac{(it\cdot x)^j}{j!}.
\]
The hypothesis gives, for every coordinate $i$,
\begin{equation}\label{eq:taylor-twisted-test}
0=\ell(\partial_iP_N-(\partial_iS)P_N)
=\int u(\partial_iP_N-(\partial_iS)P_N)\,d\nu.
\end{equation}
We have the pointwise bounds
\[
|P_N(x)|\le e^{|t||x|},
\qquad
|\partial_iP_N(x)|\le |t_i|e^{|t||x|}.
\]
Using \eqref{eq:expdom} with $Q=1$ and $Q=\partial_iS$ gives an integrable dominating function for \eqref{eq:taylor-twisted-test}.  Dominated convergence therefore yields
\begin{equation}\label{eq:fourier-twisted-test}
\int u(x)(it_i-\partial_iS(x))e^{it\cdot x}\,d\nu(x)=0.
\end{equation}
Writing $v(x)=Z_S^{-1}u(x)e^{-S(x)}$, the same Gaussian-tail
estimate shows that both $v$ and $(\partial_iS)v$ belong to
$L^1(dx)$. With the Fourier convention
\[
\widehat v(t)=\int e^{it\cdot x}v(x)\,dx,
\]
we have $\widehat{\partial_i v}(t)=-it_i\widehat v(t)$. Multiplying \eqref{eq:fourier-twisted-test} by $-1$ therefore gives
\[
\widehat{\partial_iv+(\partial_iS)v}(t)=0.
\]
Both terms define tempered distributions, so Fourier injectivity gives
\begin{equation}\label{eq:distributional-twisted-equation}
\partial_iv+(\partial_iS)v=0
\quad\text{in }\mathcal D'(\R^d)
\end{equation}
for every $i$.

\medskip\noindent\textit{Step 3. Identification of the functional.}
On any relatively compact open set, we may multiply the distributional
identity \eqref{eq:distributional-twisted-equation} by the smooth
function $e^S$, and the product rule gives
\[
\partial_i(e^Sv)
=e^S(\partial_iv+(\partial_iS)v)=0.
\]
Thus all first distributional derivatives of $e^Sv$ vanish. A distribution on connected $\R^d$ with zero gradient is constant, since convolution on each ball with a compactly supported approximate identity is smooth and has zero gradient on a smaller ball, hence is constant there. Constants agree on overlapping balls, and passage to the distributional limit shows $e^Sv=C_0$
for some $C_0\in\C$.  Since $v=Z_S^{-1}ue^{-S}$, $u=Z_SC_0$ is constant almost everywhere.  Therefore
\[
\ell(P)=c\int P\,d\nu
\]
with $c=Z_SC_0$. Since the representing density $u$ is constant, the same identity holds for every $F\in L^p(\nu)$, giving
\[
\widetilde\ell(F)=c\int F\,d\nu.
\]
Evaluating at $1$ gives $c=\ell(1)$, which is independent of the chosen extension. Since the argument applies to every continuous extension, they all coincide, proving the theorem without presupposing polynomial density.
\end{proof}

\begin{corollary}[Finite saturation]\label{cor:finitesaturation}
With $D_S=\sum_i(\partial_i-\partial_iS)\C[x]$, one has
\[
\overline{D_S}^{\,L^p(\nu)}=\ker\left(F\mapsto\int F\,d\nu\right).
\]
In standard Gaussian coordinates, write $S=|x|^2/2+f$ and $w_f=e^{-f}$. Let $a_{f,\C}$ be the complex polynomial-class map for this action, with the Gaussian normalization of Section~\ref{subsec:normalized}. Then
\begin{equation}\label{eq:weightedclosure}
\overline{w_f\ker a_{f,\C}}^{\,L^1(\gamma_d)}=\ker\left(F\mapsto\int F\,d\gamma_d\right).
\end{equation}
\end{corollary}
\begin{proof}
Let $B_iP=\partial_iP-(\partial_iS)P$.
Coercivity implies that $Pe^{-S}$ and $\partial_i(Pe^{-S})$ are integrable.  Fubini reduces the integral in the $x_i$-direction to a one-dimensional boundary term, and the Gaussian decay makes that boundary term zero.  Hence
\[
\int B_iP\,d\nu=0.
\]
Therefore
\[
D_S:=\sum_iB_i\C[x]
\subseteq\ker I_\nu,
\qquad I_\nu(F)=\int F\,d\nu.
\]
By continuity of $I_\nu$, also
\begin{equation}\label{eq:saturation-easy-inclusion}
\overline{D_S}^{\,L^p(\nu)}\subseteq\ker I_\nu.
\end{equation}

If the inclusion in \eqref{eq:saturation-easy-inclusion} were
strict, we could choose
\[
F_0\in\ker I_\nu\setminus\overline{D_S}.
\]
By Hahn--Banach there is a continuous linear functional $L$ on $L^p(\nu)$ such that
\[
L|_{\overline{D_S}}=0,
\qquad L(F_0)\ne0.
\]
The restriction of $L$ to polynomials is $L^p$-bounded and annihilates every twisted derivative. Theorem~\ref{thm:finiterigidity} identifies this restriction with a scalar multiple of integration and identifies its unique continuous extension to $L^p(\nu)$. Since $L$ is itself such an extension,
\[
L(F)=c\int F\,d\nu
\]
for all $F\in L^p(\nu)$.  Since $F_0$ has mean zero, this gives $L(F_0)=0$, contradiction.  Thus
\[
\overline{D_S}^{\,L^p(\nu)}=\ker I_\nu.
\]
Now write
\[
S(x)=\frac{|x|^2}{2}+f(x),
\qquad
w_f=e^{-f},
\qquad
Z_f=\int w_f\,d\gamma_d, \qquad d\nu=Z_f^{-1}w_f\,d\gamma_d.
\]
The multiplication map $M_{w_f}:L^1(\nu)\to L^1(\gamma_d)$, $F\mapsto w_fF$,
satisfies
\[
\|M_{w_f}F\|_{L^1(\gamma_d)}
=Z_f\|F\|_{L^1(\nu)}.
\]
Since $w_f>0$ almost everywhere, the inverse is $G\mapsto G/w_f$, making this map a Banach-space isomorphism up to the constant factor $Z_f$.  Moreover
\[
\int F\,d\nu=0
\iff
\int M_{w_f}F\,d\gamma_d=0.
\]
By Lemma~\ref{lem:top}, after complexification $\ker a_{f,\C}=D_S$. Applying the homeomorphism $M_{w_f}$ to the equality already proved for $p=1$ gives exactly
\[
\overline{w_f\ker a_{f,\C}}^{\,L^1(\gamma_d)}
=\ker\left(G\mapsto\int G\,d\gamma_d\right).
\]
This concludes the proof.
\end{proof}

\subsection{Cumulative realization and saturation}\label{subsec:cumulative}
Fix a tower of levels in $\mathcal T$ whose nested Gaussian spaces $K_1\subset K_2\subset\cdots$ generate the field sigma-algebra. We may take $K_m=H_{n,m}$ with $n$ fixed or a tower from Proposition~\ref{prop:diagonal}. Write $f_m$, $N_m$, and $A_m$ for the interaction, normalized motive, and polynomial algebra at the $m$th level. Let $\mu_m$ be its standard Gaussian marginal and put
\[
E_m=L^1(\mu_m;\C),\qquad E=L^1(\mu_C;\C),\qquad w_m=e^{-f_m}.
\]
The isometric embedding $J_m:E_m\to E$ views a function of the coordinates in $K_m$ as a function on the full field. With $a_{m,\C}$ denoting the complexified polynomial-class map at level $m$, define
\[
\mathcal R_{\mathrm{phys}}
=\Span_{\C}\bigcup_{m\ge1}J_m\bigl(w_m\ker a_{m,\C}\bigr)
\subseteq E.
\]
These finite sums may combine relations from different cutoffs and have zero integral by \eqref{eq:finiteperiod}, and the theorem below shows that they are dense in the mean-zero subspace. To keep track simultaneously of their classes and their integrands, put
\begin{equation}\label{eq:cumulative-data}
M_\infty=\bigoplus_{i\ge1}N_i\quad\text{in }\Ind(\Cat),
\qquad
D_\infty=\bigoplus_{i\ge1}A_i.
\end{equation}
In the ind-category, this direct sum is the filtered colimit of the partial sums $\bigoplus_{i\le m}N_i$ under inclusion. No maps $N_i\to N_{i+1}$ are needed because $M_\infty$ keeps the levels separate rather than identifying a motive of their constructive limit.

The realization functors and comparison extend term by term.
\[
\begin{gathered}
\omega_\B(M_\infty)=\bigoplus_i\omega_\B(N_i),\qquad
\omega_{\dR}(M_\infty)=\bigoplus_i\omega_{\dR}(N_i),\\
c_{M_\infty}=\bigoplus_i c_{N_i}.
\end{gathered}
\]
For the last equality, complexify over $k$ on the de Rham side and over $\Q$ on the Betti side. All vectors have finite support, unlike a general fixed-observable sequence, whose weighted representatives may converge in $E$ without making the sequence a vector of the direct sum.

Write $D_{\infty,\C}=D_\infty\otimes_k\C$. Define the class map and the integrand map by
\[
\begin{aligned}
\alpha_\infty&=\bigoplus_i a_i:D_\infty\longrightarrow\omega_{\dR}(M_\infty),\\
j_\infty((P_i))&=\sum_iJ_i(P_iw_i)\in E,
\end{aligned}
\]
with $j_\infty$ defined on $D_{\infty,\C}$. Each $a_i$ is surjective by Lemma~\ref{lem:top}, so $\alpha_\infty$ is surjective. Moreover,
\[
\ker\alpha_{\infty,\C}=\bigoplus_i\ker a_{i,\C},
\qquad
j_\infty(\ker\alpha_{\infty,\C})=\mathcal R_{\mathrm{phys}}.
\]
Put
\[
\beta_\infty((v_i))=\sum_i\beta_i(v_i),
\qquad
I_\infty(F)=\int F\,d\mu_C.
\]
The sum defining $\beta_\infty$ is finite on each vector of $\omega_\B(M_\infty)$. Summing the finite period identity~\eqref{eq:finiteperiod} gives the commutative diagram
\[
\begin{tikzcd}[column sep=2em,row sep=large]
D_{\infty,\C} \arrow[r,"\alpha_{\infty,\C}"] \arrow[d,"j_\infty"']
& \omega_{\dR}(M_\infty)_{\C} \arrow[r,"c_{M_\infty}"]
& \omega_{\B}(M_\infty)_{\C} \arrow[r,"\beta_{\infty,\C}"]
& \C \arrow[d,equal] \\
E \arrow[rrr,"I_\infty"'] &&& \C
\end{tikzcd}
\]
and hence
\begin{equation}\label{eq:presentationidentity}
I_\infty j_\infty
=\beta_{\infty,\C}c_{M_\infty}\alpha_{\infty,\C}.
\end{equation}

\begin{theorem}[Saturation of the cumulative presentation]\label{thm:saturation}
For the cumulative presentation of Section~\ref{subsec:cumulative},
\begin{equation}\label{eq:globalsaturation}
\overline{j_\infty(\ker\alpha_{\infty,\C})}^{\,E}
=\ker I_\infty.
\end{equation}
The quotient $E/\overline{j_\infty(\ker\alpha_{\infty,\C})}$ is isometrically isomorphic to $\C$ through integration.
\end{theorem}
\begin{proof}
If $d\in\ker\alpha_{\infty,\C}$, the presentation identity gives
\[
I_\infty j_\infty(d)
=\beta_{\infty,\C}c_{M_\infty}\alpha_{\infty,\C}(d)=0.
\]
Hence $j_\infty(\ker\alpha_{\infty,\C})\subseteq\ker I_\infty$. The kernel of the continuous functional $I_\infty$ is closed, so
\begin{equation}\label{eq:global-saturation-easy-inclusion}
\overline{j_\infty(\ker\alpha_{\infty,\C})}^{\,E}
\subseteq\ker I_\infty.
\end{equation}

For the reverse inclusion, consider a mean-zero
$F\in E=L^1(\mu_C)$ and its conditional expectations
$F^{(m)}=\E[F\mid\mathcal F_m]$, where
$\mathcal F_m=\sigma(X(K_m))$ is completed. Since these
sigma-algebras increase to the completed field sigma-algebra,
the upward convergence theorem
\cite[Chap.~9]{Kallenberg} gives
\begin{equation}\label{eq:conditional-l1-convergence}
F^{(m)}\to F\quad\text{in }L^1(\mu_C).
\end{equation}
Conditional expectation preserves the mean, so
\begin{equation}\label{eq:conditional-zero-mean}
\E[F^{(m)}]=0.
\end{equation}
The isometry $J_m:E_m\to E$ identifies $E_m$ with the $\mathcal F_m$-measurable $L^1$ subspace, so there is a unique $\widetilde F_m\in E_m$ with $J_m\widetilde F_m=F^{(m)}$, and \eqref{eq:conditional-zero-mean} gives it zero $\mu_m$-mean.

For a given $\varepsilon>0$, the weighted Gaussian form
\eqref{eq:weightedclosure} of
Corollary~\ref{cor:finitesaturation} provides a polynomial
$P_m$ such that
\begin{equation}\label{eq:finite-exact-approximation}
a_{m,\C}(P_m)=0,
\qquad
\|\widetilde F_m-P_mw_m\|_{E_m}<\varepsilon.
\end{equation}
Let $d_m\in D_{\infty,\C}$ have $P_m$ in its $m$th summand and zero in all other components, so that $\alpha_{\infty,\C}(d_m)=0$ and $j_\infty(d_m)=J_m(P_mw_m)$.
Since $J_m$ is isometric, \eqref{eq:finite-exact-approximation} becomes
\begin{equation}\label{eq:global-exact-approximation}
\|F^{(m)}-j_\infty(d_m)\|_E<\varepsilon.
\end{equation}
First choose $m$ so large that $\|F-F^{(m)}\|_1<\varepsilon$ by \eqref{eq:conditional-l1-convergence}, and then choose $P_m$ so that \eqref{eq:global-exact-approximation} holds.  The triangle inequality gives
\[
\|F-j_\infty(d_m)\|_1<2\varepsilon.
\]
Thus $F$ lies in the closure of $j_\infty(\ker\alpha_{\infty,\C})$.  Together with \eqref{eq:global-saturation-easy-inclusion}, this proves \eqref{eq:globalsaturation}.

Integration has operator norm one on $L^1(\mu_C)$ and satisfies
\[
|I_\infty(F)|\le\|F\|_1.
\]
By the equality of kernels just proved, it induces an injective map
\[
\overline I:E/\overline{j_\infty(\ker\alpha_{\infty,\C})}\to\C.
\]
For every $z\in\C$, the constant function $z$ has integral $z$
and $L^1$ norm $|z|$, proving surjectivity and giving the upper
bound $|z|$ for the quotient norm. The inequality
$|I_\infty(F)|\le\|F\|_1$ gives the reverse bound, so
$\overline I$ is an isometric isomorphism.
\end{proof}

Because the proof uses only finite saturation and generation of the field sigma-algebra, rather than convergence of $f_m$ or $w_m$, the theorem applies both at fixed ultraviolet regularization and along an admissible diagonal.

\paragraph{Continuous covectors and the Galois action.}
Let $\gamma$ be a complex-linear covector on $\omega_\B(M_\infty)_\C$ and suppose that
\begin{equation}\label{eq:globalbound}
|\gamma c_{M_\infty}\alpha_{\infty,\C}(d)|
\le C\norm{j_\infty(d)}_1\qquad(d\in D_{\infty,\C})
\end{equation}
for some $C<\infty$. The bound uses the norm of the \emph{sum} of the weighted integrands and must therefore respect cancellations between levels, so separate bounds at each cutoff do not suffice. It makes
\begin{equation}\label{eq:continuous-covector-functional}
L_0(j_\infty(d)):=\gamma c_{M_\infty}\alpha_{\infty,\C}(d)
\end{equation}
well defined and continuous on $j_\infty(D_{\infty,\C})$ by \eqref{eq:globalbound}. Hahn--Banach extends $L_0$ continuously to $E$. Every such extension $L$ annihilates $j_\infty(\ker\alpha_{\infty,\C})$, and hence, by Theorem~\ref{thm:saturation}, all of $\ker I_\infty$. Therefore
\begin{equation}\label{eq:continuous-functional-scalar}
L=cI_\infty
\end{equation}
for some $c\in\C$. Evaluating on $j_\infty(d)$ and using \eqref{eq:presentationidentity}, the surjectivity of $\alpha_\infty$, and the comparison isomorphism gives $\gamma=c\beta_{\infty,\C}$.

For $\sigma\in G_{\dR}(\C)$, write $\sigma_{M_\infty}$ for its termwise action on the complex de Rham realization. Its transformed periods are obtained by replacing the physical Betti covector with
\[
\gamma_\sigma
=\beta_{\infty,\C}c_{M_\infty}\sigma_{M_\infty}c_{M_\infty}^{-1}.
\]
Indeed, $\gamma_\sigma c_{M_\infty}=\beta_{\infty,\C}c_{M_\infty}\sigma_{M_\infty}$. In this cumulative presentation, \eqref{eq:globalbound} is exactly the condition that the transformed covector define a bounded functional of the physical integrand. Whenever it holds, the preceding argument gives $\gamma_\sigma=c\beta_{\infty,\C}$ for some $c\in\C$, although the Galois action by itself supplies no such estimate.

Along an admissible diagonal, the weighted approximants of every fixed $q\in A_{\Q}$ converge to $qe^{-V_g}$ in $L^1(\mu_C)$. A bounded realization $L$ therefore implies convergence of the transformed periods and gives
\[
L(qe^{-V_g})
=\lim_r\per_{\C}(\sigma\cdot p_r^{\mathfrak m}(q))
=\ev_{\infty}(\sigma\cdot\mathbf p^{\mathfrak m}(q)).
\]
When $g\not\equiv0$, Theorem~\ref{thm:spatial-divergence} shows that $\sigma_{\mathrm{rot}}$ cannot satisfy \eqref{eq:globalbound} along an admissible diagonal.

Conversely, every scalar multiple of integration is bounded on $E$.

\begin{corollary}[The subgroup of physically continuous transformations]\label{cor:physical-subgroup}
For the chosen cumulative presentation, set
\begin{equation}\label{eq:physical-subgroup}
H_{\mathrm{phys}}
=\{\sigma\in G:\gamma_\sigma\in\C^\times\beta_{\infty,\C}\}.
\end{equation}
This is precisely the set of group points satisfying \eqref{eq:globalbound} with $\gamma=\gamma_\sigma$ and a constant that may depend on $\sigma$. It is a subgroup of $G$, and the uniquely determined scalar in
\[
\gamma_\sigma=\chi(\sigma)\beta_{\infty,\C}
\]
defines a group character on complex points, $\chi:H_{\mathrm{phys}}\to\C^\times$. Consequently, $H_{\mathrm{phys}}$ is the largest subgroup all of whose elements admit a bounded realization on the common physical $L^1$ space for this presentation.

Along an admissible diagonal, every $\sigma\in H_{\mathrm{phys}}$ satisfies
\begin{equation}\label{eq:physical-subgroup-limits}
\lim_r\per_{\C}(\sigma\cdot p_r^{\mathfrak m}(q))
=\chi(\sigma)\mcI_g(q)
\qquad(q\in A_{\Q}).
\end{equation}
The kernel of $\chi$ consists exactly of the transformations preserving all physical periods of polynomial insertions at every level of the chosen tower. If $g\not\equiv0$ and the tower is an admissible diagonal, then $\sigma_{\mathrm{rot}}\notin H_{\mathrm{phys}}$, so this subgroup is proper.
\end{corollary}
\begin{proof}
The covector $\beta_{\infty,\C}$ is nonzero because its value on the partition frame at any level, after comparison, is the positive physical partition integral. Since $c_{M_\infty}\sigma_{M_\infty}c_{M_\infty}^{-1}$ is an automorphism, $\gamma_\sigma$ is also nonzero. The scalar characterization proved above therefore shows that \eqref{eq:globalbound} implies membership in $H_{\mathrm{phys}}$, with a unique nonzero scalar $\chi(\sigma)$. Conversely, membership gives
\[
\gamma_\sigma c_{M_\infty}\alpha_{\infty,\C}(d)
=\chi(\sigma)I_\infty j_\infty(d),
\]
so \eqref{eq:globalbound} holds with $C=|\chi(\sigma)|$.

Write $\widetilde\sigma=c_{M_\infty}\sigma_{M_\infty}c_{M_\infty}^{-1}$ and abbreviate $\beta_{\infty,\C}$ to $\beta$. For $\sigma,\tau\in H_{\mathrm{phys}}$, the action convention gives
\[
\beta\widetilde\sigma\widetilde\tau
=\chi(\sigma)\beta\widetilde\tau
=\chi(\sigma)\chi(\tau)\beta,
\qquad
\beta\widetilde\sigma^{-1}=\chi(\sigma)^{-1}\beta.
\]
Thus $H_{\mathrm{phys}}$ is a subgroup and $\chi$ is a character, while the equivalence with \eqref{eq:globalbound} proves its stated maximality. Applying the bounded functional $\chi(\sigma)I_\infty$ to the convergent weighted approximants gives \eqref{eq:physical-subgroup-limits}. Surjectivity of the polynomial-class maps identifies $\ker\chi$ with preservation of every finite physical period, and Theorem~\ref{thm:spatial-divergence} excludes $\sigma_{\mathrm{rot}}$.
\end{proof}

After comparison, $H_{\mathrm{phys}}$ is the stabilizer of the line $\C\beta_{\infty,\C}$ in the algebraic dual of the direct-sum Betti realization. Its maximality concerns physical $L^1$-continuity rather than convergence alone. It depends on the full cumulative presentation. Section~\ref{sec:scope} shows that changing only the first level can change this subgroup while leaving the sequential symbols and their limiting evaluations unchanged.

The constructive functional $q\mapsto\mcI_g(q)$ and the divergent transformation $\sigma_{\mathrm{rot}}$ of Theorem~\ref{thm:spatial-divergence} are common to every admissible diagonal. Independence of the sequential symbol algebra and its coaction domain from the admissible diagonal is not established here, and no motive or Galois subgroup intrinsic to the continuum theory is asserted beyond these diagonal-independent conclusions.

\section{A quartic example of convergence without physical \texorpdfstring{$L^1$}{L1}-continuity}\label{sec:scope}
We now give a quartic example in which all Galois translates of fixed-observable sequences converge, but physical $L^1$-continuity fails for some transformations. We also show that changing only the first level can change the subgroup satisfying this condition without changing the sequential symbols or their limiting evaluations.

\subsection{A quartic model with a non-scalar limiting Galois profile}\label{subsec:quarticprofile}
Take $\Q$ as the coefficient field, let $\gamma_1$ be the standard Gaussian measure on $\R$, and put
\begin{equation}\label{eq:quarticmodel}
f(x)=\frac1{12}H_4(x)=\frac{x^4-6x^2+3}{12},
\qquad S(x)=\frac{x^2}{2}+f(x)=\frac{x^4}{12}+\frac14.
\end{equation}
Here $H_4$ is the Wick--Hermite polynomial in \eqref{eq:wickquartic}. The choice of coupling cancels the quadratic term, leaving explicitly computable periods. For $P\in\Q[x]$, write
\[
I(P)=\int_{\R}P(x)e^{-f(x)}\,d\gamma_1(x)
=\frac{e^{-1/4}}{\sqrt{2\pi}}\int_{\R}P(x)e^{-x^4/12}\,dx.
\]
It suffices to compare $Z=I(1)$ with $M_2=I(x^2)$ because proportionality to physical integration would require the same factor for both. Here $Z$ is the partition function and $M_2/Z$ the normalized second moment.

The finite normalized motive of Section~\ref{subsec:normalized} for this integral is
\[
N_*=H^1_{\exp}(\A^1_{\Q},S)\otimes G_0^\vee,
\qquad G_0=H^1_{\exp}(\A^1_{\Q},x^2/2),
\]
where $G_0$ is the Gaussian line. With the insertion class $a_*(P)$ and real integration covector $\beta_*$, we have
\[
p_*(P)=[N_*,\beta_*,a_*(P)],\qquad
\per(p_*(P))=I(P).
\]
Put $p_0=p_*(1)$ and $p_2=p_*(x^2)$ in $P_*^{\mathfrak m}=\mathcal O(T_*)$, where $T_*$ is the comparison torsor of $\Cat_*=\langle G_0,N_*\rangle_\otimes$, and write $G_{*,\dR}$ for its de Rham group.

The scaling monodromy of Lemma~\ref{lem:scaling-monodromy} gives an explicit non-scalar transform and proves the independence of the two even symbols. The gamma evaluations below also yield an arithmetic refinement.

\begin{proposition}\label{prop:quarticprofile}
For the model \eqref{eq:quarticmodel}, the following statements hold.
\begin{enumerate}
\item The de Rham realization of $N_*$ has dimension three over $\Q$, with basis $a_*(1)$, $a_*(x)$, and $a_*(x^2)$. When period symbols are formed using the real integration covector $\beta_*$, the odd class $a_*(x)$ has zero symbol, whereas $p_0=p_*(1)$ and $p_2=p_*(x^2)$ are linearly independent over $\Q$. Consequently, every symbol $p_*(P)$, with $P\in\Q[x]$, is a unique $\Q$-linear combination of $p_0$ and $p_2$. The space
\[
W_*=\{p_*(P):P\in\Q[x]\}
=\Q p_0\oplus\Q p_2
\]
is stable under the coaction of $\mathcal O(G_{*,\dR})$. In the terminology of Theorem~\ref{thm:profiles}, the visible quotient $N_*^{\mathrm{vis}}$ therefore has rank two.
\item Writing $Z=\per(p_0)$ and $M_2=\per(p_2)$, one has
\begin{equation}\label{eq:quarticgamma}
\begin{aligned}
Z&=\frac{e^{-1/4}}{\sqrt{2\pi}}\frac{12^{1/4}}2\Gamma(1/4),\\
M_2&=\frac{e^{-1/4}}{\sqrt{2\pi}}\frac{12^{3/4}}2\Gamma(3/4),
\end{aligned}
\qquad
\frac{M_2}{Z}=\frac{\sqrt{24}\,\pi}{\Gamma(1/4)^2}\notin\overline{\Q}.
\end{equation}
\item For $\sigma\in G_{*,\dR}(\C)$ set
\[
Z_\sigma=\per_{\C}(\sigma\cdot p_0),\qquad
M_{2,\sigma}=\per_{\C}(\sigma\cdot p_2).
\]
The construction of Lemma~\ref{lem:scaling-monodromy} gives a point $\sigma_*\in G_{*,\dR}(\C)$ satisfying
\[
Z_{\sigma_*}=iZ,\qquad M_{2,\sigma_*}=-iM_2.
\]
Thus, for $\sigma=\sigma_*$,
\begin{equation}\label{eq:quarticdeterminant}
Z M_{2,\sigma}-M_2 Z_\sigma\ne0.
\end{equation}
In particular, the two physical periods do not transform by a common scalar for every group point.
\end{enumerate}
\end{proposition}
\begin{proof}
\medskip\noindent\textit{Step 1. De Rham quotient and parity.}
For $S(x)=\frac{x^4}{12}+\frac14$ one has $S'(x)=\frac{x^3}{3}$,
so the twisted derivative operator is
\[
BP=P'-\frac{x^3}{3}P.
\]
Applying this operator to $P=1$ gives $[x^3dx]=0$, while
applying it to $P=x^n$, $n\ge1$, gives
\[
0=[nx^{n-1}-x^{n+3}/3]dx,
\]
hence the recurrence $[x^{n+3}dx]=3n[x^{n-1}dx]$.
Repeated use reduces every monomial class to a linear combination of $[dx]$, $[xdx]$ and $[x^2dx]$, so the de Rham rank is at most three.

To prove independence, suppose $BP$ has degree at most two.  If $P\ne0$ has degree $d$, then the term $-(x^3/3)P$ has degree $d+3$, whereas $P'$ has degree $d-1$.  These degrees cannot cancel, so
\[
\deg BP=d+3\ge3.
\]
No nonzero polynomial of degree at most two can therefore lie in
the image of $B$, so the three displayed classes are independent
and the de Rham rank is exactly three. This rank is unchanged by
tensoring with the one-dimensional Gaussian dual.

Reflection $\iota:x\mapsto-x$ preserves both the interacting action and the Gaussian action.  On the normalized motive, the sign of the one-dimensional volume form is cancelled by the sign of the Gaussian dual frame, so
\[
\iota a_*(P)=a_*(P(-x)).
\]
Change of variables on the real integration cycle gives
\[
\beta_*\iota=\beta_*.
\]
The odd subspace of the three-dimensional de Rham realization, spanned by $a_*(x)$, is the realization of a rank-one motivic subobject annihilated by $\beta_*$.  Consequently every odd framed symbol vanishes, so
\[
p_*(P)=0\qquad(P(-x)=-P(x)).
\]
The recurrence above reduces every even polynomial class to a combination of $a_*(1)$ and $a_*(x^2)$.  Hence
\[
W_*\subseteq\Q p_0+\Q p_2,
\]
and the visible rank is at most two.

\medskip\noindent\textit{Step 2. Evaluation of the even periods.}
For $j=0,1$, evenness gives
\[
\int_{\R}x^{2j}e^{-x^4/12}\,dx
=2\int_0^\infty x^{2j}e^{-x^4/12}\,dx.
\]
Under the substitution $t=x^4/12$, we have
\[
x=12^{1/4}t^{1/4},
\qquad
dx=\frac{12^{1/4}}4t^{-3/4}\,dt,
\]
and therefore
\[
2x^{2j}dx
=\frac{12^{(2j+1)/4}}2
 t^{(2j+1)/4-1}\,dt.
\]
Thus
\[
\int_{\R}x^{2j}e^{-x^4/12}\,dx
=\frac{12^{(2j+1)/4}}2
\Gamma\left(\frac{2j+1}{4}\right).
\]
Multiplying by $e^{-1/4}(2\pi)^{-1/2}$, which accounts for the
constant term in the action and the Gaussian normalization, gives
the formulas for $Z$ and $M_2$ in \eqref{eq:quarticgamma}.

Euler's reflection identity
\[
\Gamma(1/4)\Gamma(3/4)=\frac{\pi}{\sin(\pi/4)}=\pi\sqrt2
\]
gives
\[
\frac{M_2}{Z}
=\sqrt{12}\,\frac{\Gamma(3/4)}{\Gamma(1/4)}
=\frac{\sqrt{24}\,\pi}{\Gamma(1/4)^2}.
\]
Chudnovsky proved the algebraic independence of $\pi$ and $\Gamma(1/4)$ \cite{Chudnovsky}. The same conclusion follows from the specialization of Nesterenko's stronger theorem stated on p.~8 of the arXiv version of \cite{Fonseca}, immediately after Theorem~3.2 (the original source is \cite{Nesterenko}). Suppose for contradiction that
\[
r:=\frac{M_2}{Z}\in\overline{\Q}.
\]
Then
\[
r\Gamma(1/4)^2-\sqrt{24}\,\pi=0
\]
is a nonzero polynomial relation between $\Gamma(1/4)$ and $\pi$ with algebraic coefficients.  A relation over a finite algebraic extension can be converted to a nonzero relation over $\Q$ by taking the product over Galois conjugates of its coefficients.  This contradicts algebraic independence and proves that $M_2/Z$ is transcendental.

\medskip\noindent\textit{Step 3. Monodromy and linear independence.}
Applying Lemma~\ref{lem:scaling-monodromy} over $\Q$, restricting
to $\Cat_*$, and transporting the inverse of the forward Betti
monodromy through comparison as in \eqref{eq:rotation-galois-point} gives
$\sigma_*\in G_{*,\dR}(\C)$. Along $t=e^{i\theta}$, the contours
$\Gamma_\theta=e^{-i\theta/4}\R$ satisfy
\[
e^{i\theta}S(e^{-i\theta/4}x)=\frac{x^4}{12}+\frac{e^{i\theta}}4,
\]
and therefore remain in rapid-decay directions throughout the loop. At $\theta=2\pi$, one has $S(-ix)=S(x)$, and the volume form contributes $-i$. The Gaussian factor has monodromy $-1$, as in the proof of Theorem~\ref{thm:spatial-divergence}. Thus, extending $I$ complex-linearly,
\begin{equation}\label{eq:quartic-rotation}
\per_{\C}(\sigma_*\cdot p_*(P))
=\frac{-i}{-1}I(P(-ix))=iI(P(-ix)).
\end{equation}
Taking $P=1$ and $P=x^2$ gives the stated values $iZ$ and $-iM_2$. Since $Z,M_2>0$,
\[
Z M_{2,\sigma_*}-M_2 Z_{\sigma_*}=-2iZM_2\ne0,
\]
which proves \eqref{eq:quarticdeterminant}.

The same transformation proves independence of the symbols. If $up_0+vp_2=0$ with $u,v\in\Q$, evaluation at the physical comparison and after applying $\sigma_*$ gives
\[
uZ+vM_2=0,\qquad iuZ-ivM_2=0.
\]
Hence $u=v=0$, and
\[
W_*=\Q p_0\oplus\Q p_2.
\]
Theorem~\ref{thm:profiles} identifies the kernel of the symbol map with the de Rham realization of the maximal invisible subobject, so the dimension of this framed span is exactly the rank of the visible quotient. Hence
\[
\operatorname{rank}N_*^{\mathrm{vis}}=2.
\]
The fixed-covector coaction formula \eqref{eq:coaction} shows that the span of all $[N_*,\beta_*,v]$ is stable under the finite de Rham coaction, making $W_*$ an ordinary finite-dimensional $\mathcal O(G_{*,\dR})$-comodule.
\end{proof}

\paragraph{From the finite motive to a sequential example.}
Adjoin independent Gaussian variables while keeping the interaction in the first coordinate, so that the action at level $r$ is $S(x_1)+\frac12\sum_{j=2}^r x_j^2$. Writing $N_r$ for its normalized motive, Proposition~\ref{prop:spectators} gives
\[
N_r\simeq N_*\qquad(r\ge1).
\]
Once a level contains all variables of a fixed observable, integrating out the spectators gives the same one-variable symbol at every later level. For instance, $X_1^2X_2^2$ gives $p_2$ from level $2$ onward because $\E[X_2^2]=1$.

\begin{corollary}[A stabilized sequential example]\label{cor:quarticsequence}
Let $H_*=\ell^2(\N)$ with orthonormal basis $(e_j)_{j\ge1}$, let $X$ be the associated isonormal Gaussian process in the sense of Section~\ref{subsec:gaussian-toolbox}, and put $X_j=X(e_j)$. Thus $(X_j)_{j\ge1}$ are independent standard Gaussian variables. Take the fixed observable algebra $\Q[X_1,X_2,\ldots]$, and at level $r$ use
\begin{equation}\label{eq:quartictower}
\begin{aligned}
K_r&=\Span(e_1,\ldots,e_r), & f_r&=f(X_1),\\
S_r(x)&=\frac{x_1^4}{12}+\frac14+\frac12\sum_{j=2}^r x_j^2,
& x&=(x_1,\ldots,x_r).
\end{aligned}
\end{equation}
Here $S_r=\frac12\sum_{j=1}^r x_j^2+f(x_1)$, with the displayed form following from the cancellation in \eqref{eq:quarticmodel}. Under the canonical spectator identifications of Proposition~\ref{prop:spectators}, every fixed-observable symbol sequence is eventually constant. The resulting sequential algebra is
\[
B_*=\Q[p_0,p_2]\subset P_*^{\mathfrak m},
\]
embedded as constant sequences in the common category $\Cat_*$ and equipped with an ordinary coaction of $\mathcal O(G_{*,\dR})$. After extension to $\C$, every Galois translate has a convergent numerical period sequence, and the limiting profiles recover those of Proposition~\ref{prop:quarticprofile}. In particular, at $\sigma_*$ the limits associated with $1$ and $X_1^2$ are not a common scalar multiple of their physical values.
\end{corollary}
\begin{proof}
For a fixed observable $q\in\Q[X_1,X_2,\ldots]$, choose $m$
so that $q$ depends only on $X_1,\ldots,X_m$.
At every level $r\ge m$, these coordinates belong to $K_r$,
and the action consists of the fixed quartic action in $X_1$
together with independent Gaussian spectators
$X_2,\ldots,X_r$. Proposition~\ref{prop:spectators} therefore
identifies the normalized framed symbol of $q$ with the
one-dimensional symbol
\begin{equation}\label{eq:spectator-observable-reduction}
p_*(Q_q),
\qquad
Q_q(X_1)=\E[q\mid X_1].
\end{equation}
Because the spectator variables are independent standard Gaussians, conditional expectation replaces each spectator monomial by its Gaussian moment, which is an integer for even powers and zero for odd powers. The rational coefficients of $q$ therefore give $Q_q\in\Q[X_1]$.
Proposition~\ref{prop:quarticprofile} gives
\[
p_*(Q_q)\in W_*=\Q p_0\oplus\Q p_2.
\]
Thus every fixed-observable sequence is eventually constant with value in this two-dimensional space.

Conversely, the fixed observables $1$ and $X_1^2$ give the constant sequences with values $p_0$ and $p_2$. Since constant sequences embed into the quotient by eventual equality, the fixed-observable sequential algebra is $B_*=\Q[p_0,p_2]\subset P_*^{\mathfrak m}$. Here $\Q[p_0,p_2]$ denotes the generated subalgebra and does not assert algebraic independence of the generators.

The finite-dimensional space $W_*$ is an $\mathcal O(G_{*,\dR})$-comodule by Proposition~\ref{prop:quarticprofile}. Since the coaction is an algebra homomorphism, it restricts to the generated algebra as
\[
B_*\longrightarrow B_*\otimes_\Q\mathcal O(G_{*,\dR}).
\]
Under the spectator identifications, this same coaction occurs at every sufficiently large cutoff. For each element of $B_*$, its tensor expression therefore uses finitely many de Rham coefficients independent of the cutoff, and its first tensor factors belong to $B_*$. This gives an ordinary sequential coaction of the kind considered in Section~\ref{sec:seqgalois}.

After extending scalars, write
$B_{*,\C}=B_*\otimes_\Q\C=\C[p_0,p_2]$.
This algebra is stable under $G_{*,\dR}(\C)$, and each of its
elements is represented by a constant sequence. For any fixed
$p\in B_{*,\C}$ and $\sigma\in G_{*,\dR}(\C)$, the translated
sequence is therefore constant with value $\sigma\cdot p$,
so its numerical periods converge to
$\per_{\C}(\sigma\cdot p)$.

Consequently the limiting profile map of Section~\ref{sec:seqgalois}, defined on all of $B_{*,\C}$ in this example, is
\[
\Phi_{*,\infty}(\mathbf p)(\sigma)=\per_{\C}(\sigma\cdot p),
\]
where $\mathbf p$ denotes the constant sequential class of $p$. In particular, the profiles of $\mathbf p_0$ and $\mathbf p_2$ have values $Z_\sigma$ and $M_{2,\sigma}$. At $\sigma=\sigma_*$, Proposition~\ref{prop:quarticprofile} gives
\[
Z\Phi_{*,\infty}(\mathbf p_2)(\sigma)
-M_2\Phi_{*,\infty}(\mathbf p_0)(\sigma)
=ZM_{2,\sigma}-M_2Z_\sigma\ne0.
\]
The nonzero determinant excludes a common scalar factor relating the two limits to $Z$ and $M_2$.
\end{proof}

To verify the failure of physical $L^1$-continuity on the single
interacting variable, let $\sigma$ satisfy \eqref{eq:quarticdeterminant} and consider the functional
\[
\ell_\sigma(P)=\per_{\C}(\sigma\cdot p_*(P)),
\qquad P\in\C[x],
\]
using the complex-linear extension of the symbol map. This functional annihilates every $P'-x^3P/3$. If there were a constant $C_\sigma$ such that
\[
|\ell_\sigma(P)|\le C_\sigma\norm{Pe^{-f}}_{L^1(\gamma_1)}
\quad\text{for every }P\in\C[x],
\]
then Theorem~\ref{thm:finiterigidity}, applied to $d\nu=Z^{-1}e^{-f}d\gamma_1$, would make $\ell_\sigma$ a scalar multiple of physical integration. Since $S=x^4/12+1/4$ satisfies the required quadratic lower bound, this contradicts \eqref{eq:quarticdeterminant}. Thus no such constant $C_\sigma$ exists even on polynomials of the single interacting variable, before any compatibility across cutoffs enters.

We next examine how $H_{\mathrm{phys}}$ depends on the initial
levels by considering $\tau=\sigma_*^2$ for the tower \eqref{eq:quartictower}. Transporting the quartic contour through two turns, $0\le\theta\le4\pi$, takes $e^{-i\theta/4}\R$ to $-\R$ with volume factor $-1$, while the Gaussian normalization contour $e^{-i\theta/2}\R$ returns with factor $1$. Since $S$ is even, the resulting period is
\[
\per_{\C}(\tau\cdot p_*(P))=-I(P(-x))=-I(P),
\qquad P\in\C[x].
\]
Spectator cancellation gives the same factor $-1$ for every polynomial insertion at every level of \eqref{eq:quartictower}. Thus $\tau\in H_{\mathrm{phys}}$ for this tower, with $\chi(\tau)=-1$ and bounded realization $-I_\infty$.

Now keep the spaces $K_r$ and change only the first interaction by setting
\[
f'_1=0,\qquad f'_r=f(X_1)\quad(r\ge2).
\]
The first normalized motive is now the unit, so the common category remains $\Cat_*$ and the acting group remains $G_{*,\dR}$. All fixed-observable symbol sequences agree with the original ones from level $2$ onward. Consequently, the sequential algebra $B_*$, its coaction, and all its limiting Galois evaluations are unchanged.

At the new first level the normalized partition symbol is $1$, which is fixed by $\tau$, so a scalar as in \eqref{eq:physical-subgroup} would have to equal $1$. At every later level the physical partition period is $Z>0$ and its $\tau$-transform is $-Z$, forcing that scalar to equal $-1$. Hence $\tau$ does not belong to $H_{\mathrm{phys}}$ for the modified tower. In particular, this subgroup is an invariant of the full physical presentation, not of the sequential algebra alone.

Requiring physical $L^1$-continuity only from some sufficiently
large level onward removes this dependence on finitely many
initial levels. Saturation applies to each remaining tail and
still forces the limiting evaluations to be scalar multiples of
physical integration whenever the weighted approximants converge
in $L^1$. For the spatial model,
independence from the choice of admissible diagonal remains open.

\subsection{Motivic comparisons and open questions}\label{subsec:cartier}
Our separation of symbols from evaluation follows the point of view on symbolic multiple zeta values in \cite[Sec.~9]{Cartier}. It also raises a comparison problem with Brown's motivic amplitudes and their Galois coaction \cite{BrownCosmic}. Here the interaction remains exponentiated at fixed coupling, and the finite coaction comes from the comparison torsor \cite[Sec.~2]{BrownPeriods}, \cite[Sec.~8.3]{FJ}. A relation with perturbative graph periods would have to respect summation and the coaction. Equality of resummed numerical values would not by itself identify the torsors.

Connes and Marcolli construct a universal group from a Tannakian category of equisingular flat bundles associated with perturbative renormalization \cite[Secs.~13 and~16]{ConnesMarcolli}. Its relation to the groups in Brown's construction remains unclear \cite[Sec.~0.5]{BrownCosmic}, and the present work likewise leaves comparison with that category open.

\paragraph{Questions for the spatial model.}
For a nonzero spatial cutoff, Theorem~\ref{thm:spatial-divergence} excludes the full group, while Corollary~\ref{cor:physical-subgroup} identifies the subgroup selected by physical $L^1$-continuity for the chosen presentation. How large is $H_{\mathrm{phys}}$, and which subgroups admit convergent fixed-observable periods without this continuity condition? In particular, it remains to determine whether the limiting profiles can be non-scalar when the number of interacting directions grows. A further question is which subalgebras admit the finite coaction in the sequential limit, and whether the sequential algebra and its limiting profiles are independent of the chosen admissible diagonal.

\paragraph{Acknowledgments.}
The ideas that led to this work grew out of discussions with Francis Brown in September 2022 about symmetries of the path integral compatible with Wick expansion, with the aim of relating them to the cosmic Galois action on Feynman integrals. I am grateful to him for those conversations and for encouraging me to pursue the question in two dimensions.

\enlargethispage{2\baselineskip}
\paragraph{Use of AI tools.}
The research question and motivation for this work predate the author's use of AI tools. Most of the work, including Main results~\ref{main:construction} and \ref{main:saturation}, was completed without AI assistance. Main result~\ref{main:divergence} grew out of mathematical discussions between the author and ChatGPT 5.6 Sol. The author directed these discussions and takes full responsibility for the mathematical claims, proofs, and presentation.


\begin{thebibliography}{99}
\bibitem{BrownCosmic}
F. Brown, \emph{Feynman amplitudes, coaction principle, and cosmic Galois group}, Commun. Number Theory Phys. \textbf{11} (2017), 453--556.

\bibitem{BrownPeriods}
F. Brown, \emph{Notes on motivic periods}, Commun. Number Theory Phys. \textbf{11} (2017), 557--655.

\bibitem{Cartier}
P. Cartier, \emph{A mad day's work: from Grothendieck to Connes and Kontsevich. The evolution of concepts of space and symmetry}, Bull. Amer. Math. Soc. (N.S.) \textbf{38} (2001), 389--408.

\bibitem{Chudnovsky}
G. V. Chudnovsky, \emph{Algebraic independence of constants connected with exponential and elliptic functions}, Dokl. Akad. Nauk Ukrain. SSR Ser. A (1976), no.~8, 698--701 (Russian). English summary, p.~767.

\bibitem{ConnesMarcolli}
A. Connes and M. Marcolli, \emph{Renormalization, the Riemann--Hilbert correspondence, and motivic Galois theory}, in \emph{Frontiers in Number Theory, Physics, and Geometry II}, Springer, Berlin, 2007, pp.~617--713.

\bibitem{DM}
P. Deligne and J. S. Milne, \emph{Tannakian categories}, in \emph{Hodge Cycles, Motives, and Shimura Varieties}, Lecture Notes in Math., vol.~900, Springer, Berlin, 1982, pp.~101--228.

\bibitem{Folland}
G. B. Folland, \emph{Real Analysis: Modern Techniques and Their Applications}, 2nd ed., Wiley, New York, 1999.

\bibitem{Fonseca}
T. J. Fonseca, \emph{A geometric introduction to transcendence questions on values of modular forms}, in H. Movasati, \emph{Modular and Automorphic Forms \& Beyond}, Monogr. Number Theory, vol.~9, World Scientific, Singapore, 2021, Appendix~A. Also available as \href{https://arxiv.org/abs/2011.14401v1}{arXiv:2011.14401v1} (2020).

\bibitem{FJ}
J. Fres\'an and P. Jossen, \emph{Exponential motives}, manuscript, version of 11~April 2024. \url{https://javier.fresan.perso.math.cnrs.fr/expmot.pdf}.

\bibitem{GRWLattice}
F. Gasparotto, A. Rapakoulias, and S. Weinzierl, \emph{Nonperturbative computation of lattice correlation functions by differential equations}, Phys. Rev. D \textbf{107} (2023), article no.~014502.

\bibitem{GJ}
J. Glimm and A. Jaffe, \emph{Quantum Physics: A Functional Integral Point of View}, 2nd ed., Springer, New York, 1987.

\bibitem{Hien}
M. Hien, \emph{Periods for flat algebraic connections}, Invent. Math. \textbf{178} (2009), 1--22.

\bibitem{HR}
M. Hien and C. Roucairol, \emph{Integral representations for solutions of exponential Gau\ss--Manin systems}, Bull. Soc. Math. France \textbf{136} (2008), 505--532.

\bibitem{Isserlis}
L. Isserlis, \emph{On a formula for the product-moment coefficient of any order of a normal frequency distribution in any number of variables}, Biometrika \textbf{12} (1918), 134--139.

\bibitem{Ito}
K. It\^o, \emph{Multiple Wiener integral}, J. Math. Soc. Japan \textbf{3} (1951), 157--169.

\bibitem{Janson}
S. Janson, \emph{Gaussian Hilbert Spaces}, Cambridge Tracts in Math., vol.~129, Cambridge Univ. Press, Cambridge, 1997.

\bibitem{Kallenberg}
O. Kallenberg, \emph{Foundations of Modern Probability}, 3rd ed., Probab. Theory Stoch. Model., vol.~99, Springer, Cham, 2021.

\bibitem{Stein}
G. Mijoule, M. Rai\v{c}, G. Reinert, and Y. Swan, \emph{Stein's density method for multivariate continuous distributions}, Electron. J. Probab. \textbf{28} (2023), article no.~59.

\bibitem{Mityagin}
B. S. Mityagin, \emph{The zero set of a real analytic function}, Math. Notes \textbf{107} (2020), 529--530.

\bibitem{Nelson}
E. Nelson, \emph{The free Markoff field}, J. Funct. Anal. \textbf{12} (1973), 211--227.

\bibitem{Nesterenko}
Yu. V. Nesterenko, \emph{Modular functions and transcendence questions}, Sb. Math. \textbf{187} (1996), 1319--1348.

\bibitem{Nourdin}
I. Nourdin, \emph{Lectures on Gaussian approximations with Malliavin calculus}, in \emph{S\'eminaire de Probabilit\'es XLV}, Lecture Notes in Math., vol.~2078, Springer, Heidelberg, 2013, pp.~3--89. 

\bibitem{Rosen}
J. Rosen, \emph{Sequential periods of the crystalline Frobenius}, preprint, 2018. \href{https://arxiv.org/abs/1805.01885}{arXiv:1805.01885}.

\bibitem{Ryan}
R. A. Ryan, \emph{Introduction to Tensor Products of Banach Spaces}, Springer Monogr. Math., Springer, London, 2002.

\bibitem{Sampedro}
J. C. Sampedro, \emph{On the $L^p$-spaces of projective limits of probability measures}, J. Theoret. Probab. \textbf{37} (2024), 2665--2703. 

\bibitem{Simon}
B. Simon, \emph{The $P(\phi)_2$ Euclidean (Quantum) Field Theory}, Princeton Ser. Phys., Princeton Univ. Press, Princeton, NJ, 1974.

\bibitem{WeinzierlLattice}
S. Weinzierl, \emph{Correlation functions on the lattice and twisted cocycles}, Phys. Lett. B \textbf{805} (2020), article no.~135449.
\enlargethispage{2\baselineskip}
\end{thebibliography}
\end{document}